\documentclass[11pt]{article}
\usepackage{graphicx,wrapfig,lipsum}
\usepackage{amsmath,amssymb,graphicx} 
\usepackage{amsfonts}
\usepackage[margin=0.9in]{geometry} 
\usepackage{enumerate}
\usepackage[toc]{appendix}
\usepackage{hyperref}
\usepackage{fullpage}
\usepackage{tikz}
\usepackage{tcolorbox}
\usepackage{bbm}
\usepackage{comment}
\newtheorem{definition}{Definition}[section]
\newtheorem{theorem}{Theorem}
\newtheorem{proposition}{Proposition}
\newtheorem{lemma}{Lemma}
\newtheorem{remark}{Remark}
\usepackage{comment}
\usepackage{enumitem}
\usepackage[utf8]{inputenc}
\usepackage{xcolor}
\usepackage{graphicx}
\graphicspath{ {./images/} }
\allowdisplaybreaks

\newenvironment{proof}{\paragraph{Proof:}}{\hfill$\square$}
\allowdisplaybreaks
\def\es{\varepsilon}
\begin{document}
	
	
	\title{Upscaling of a reaction-diffusion-convection problem with exploding non-linear drift}

\author{Vishnu Raveendran $^{a,*}$,  Emilio N.M. Cirillo$^b$, Adrian Muntean$^a$ \\
$^a$ Department of Mathematics and Computer Science, Karlstad University, Sweden\\
$^b$ Dipartimento di Scienze di Base e Applicate per l’Ingegneria, \\ Sapienza Universit{\`a} di Roma, Italy\\
* vishnu.raveendran@kau.se}

	\date{\today} 
	\maketitle
	
	\begin{abstract}\label{abstract}
	We study a reaction-diffusion-convection problem with nonlinear drift posed in a domain with periodically arranged obstacles. The non-linearity in the drift is linked to the hydrodynamic limit of a totally asymmetric simple exclusion process (TASEP) governing a population of interacting particles crossing a domain with obstacle. Because of the imposed large drift scaling, this nonlinearity is expected to explode in the limit of a vanishing scaling parameter. As main working techniques,  we employ two-scale formal homogenization asymptotics with drift to derive  the corresponding upscaled model equations as well as the  structure of the effective transport tensors. Finally, we use Schauder's fixed point theorem as well as monotonicity arguments to study the weak solvability of the upscaled model posed in an unbounded domain.  This study wants to contribute with theoretical understanding needed when designing thin composite materials that are resistant to high velocity impacts.
	\end{abstract}
\maketitle

\section{Introduction}\label{In}
\begin{figure}[t!]
\includegraphics[width=10cm]{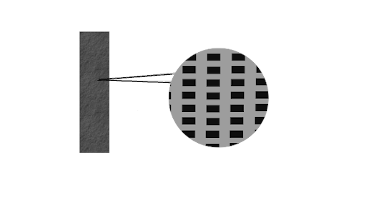}
\centering
\caption{Macroscopic view with a microscopic zoom-in of our composite layer.}
\label{fig1}
\end{figure}

\par Reaction-diffusion equations with large drift posed for  composite (porous) materials have many potential applications for real-world scenarios, like  high velocity fluid flow through composite materials, filtration combustion \cite{ijioma2019fast}, reactive flow through filters with wall integrated catalysts \cite{ILIEV2020103779}.

\par In this paper, we derive  and then analyze mathematically an upscaled equation associated to a microscopic reaction-diffusion equation with oscillating coefficients and exploding non-linear drift. The non-linearity  in the drift term is derived in an earlier work  of ours \cite{CIRILLO2016436} as hydrodynamic limit of a totally asymmetric simple exclusion process (TASEP)  for a population of interacting particles crossing a domain with obstacle. We consider the domain of definition for our problem $\Omega_\es$ as being paved by  periodically distributed  replicas of an $\es$-scaled standard cell Z  in $\mathbb{R}^n$. The standard cell $Z$ is a unit square (see Fig. \ref{scell}) with a solid rectangular obstacle placed in the centre of mass of $Z$, while $\es>0$ is a small scaling parameter linked to the multiscale structure of the material geometry (to be defined in Section \ref{MicM}). As the original discussion in \cite{CIRILLO2016436} was developed for a plannar geometry, without loss of generality, we assume  $n=2$. What concerns the target microscopic problem,  we consider that the drift is very large compare to diffusion and reaction. Consequently, the  drift will be scaled like a term of  order of $\frac{1}{\es}$, while all the other contributing effects will be of order of $\es^0$. The boundaries of the internal obstacles are assumed to have the following structure and conditions: on some part, say $\Gamma_D^\es$, we consider non-homogeneous Dirichlet boundary while on the rest of the boundary, say $\Gamma_N^\es$,  we consider  non-homogeneous Neumann boundary condition. The Hausdorff measure of $\Gamma_D^\es$ can vanish, while the Hausdorff measure of $\Gamma_N^\es$ is taken to be non-vanishing.  We consider the Dirichlet boundary term to be of order of $\es^\gamma$ ($\gamma>0$), while the Neumann boundary terms are assumed to be of order of $\es$.

\par This paper is continuation of the works \cite{raveendran21} and \cite{cirillo2020upscaling}, where we discuss similar settings posed in a thin composite layer with slow or moderate drift. The main purpose of this paper is to perform the homogenization asymptotics and analyze mathematically the upscaled equation corresponding  to the microscopic problem with the nonlinear drift exploding as $\es\rightarrow 0$. 
To cope with the presence of the large drift combined with the periodicity of the domain, we apply the method of two-scale asymptotic homogenization with drift (cf. \cite{allaire2010homogenization})   in a moving co-ordinate frame as suggested in \cite{pyatnitskiui1984averaging}, which is tailor-made for these specific asymptotic settings. The upscaled equation is then derived as a quasi-linear reaction-dispersion equation coupled with a quasi-linear elliptic cell problem. Notably, the resulting dispersion tensor compensates for both microscopic diffusion and drift mechanisms. As the upscaled equation is a quasi-linear parabolic problem posed in an unbounded domain also coupled strongly with a quasi-linear elliptic problem, ensuring its weak solvability is a challenging task. To do so, we are combining a number of technical ingredients including Schauder's fixed-point theorem (see Theorem 3 in section 9.2.2 of \cite{evans2010partial}), Kirchhoff's transformation  (see \cite{bagnall2013application}), and a monotonicity argument (see \cite{HILHORST20071118}), all these applied to an auxiliary problem that has the same structure as the upscaled equation just that is posed in a bounded smooth domain. Schauder's fixed-point argument takes care of the existence of weak solutions for the bounded domain formulation, while the Kirchhoff's transformation recasts the problem so that we can prove a positivity result as well as a comparison principle.
Extending the bounded domain solution to whole $\mathbb{R}^2$, we obtain a sequence of monotonically convergent solutions corresponding to fixed diameters of the bounded domains. We conclude that this sequence converges to the solution of the target upscaled problem posed in unbounded domain.

\par For the basic theory of homogenization, we refer the reader for instance to the classical textbook \cite{donato1999homogenization}. The method of formal two-scale asymptotic expansions with drift for the linear exploding drift case is introduced in \cite{ALLAIRE20102292}, see also \cite{allaire_orive_2007}, \cite{hutridurga}, \cite{hutridur},  and \cite{allaire2016} for related situations where the concept of two-scale convergence with drift (cf. \cite{marusik2005}) is used. A first justification for $u_\es\rightarrow u_0$ as $\es \rightarrow 0$ for asymptotic expansion with drift is given in \cite{allaire2010homogenization}, while corrector estimates and related mutiscale numerical simulations for linear reaction-diffusion problem with large drift were performed in \cite{ouaki2012multiscale} and \cite{ouaki2015priori}. Different numerical approximation strategies of the same problem were proposed in \cite{henning2010heterogeneous} relying on the concept of  heterogeneous multi-scale method (the HMM method). Treating the  same asymptotic questions for bounded domains is more troublesome and we avoid it here. We only mention in passing the Ref. \cite{ALLAIRE2012300}, where the linear case is completely solved. Many aspects are yet unexplored in the bounded domain case. A promising direction which combines homogenization with dimension reduction is dealt with in \cite{irina2011}. 
The approach is possibly applicable in our case as well. From a totally different perspective, it would be interesting to study how this type of two-scale asymptotics with drift can cope with an eventual stochasticity either in the geometry of the material (e.g. in the distribution and/or choice of shapes of the obstacles) or in the dynamics of the problem; see \cite{piatnitski2020homogenization}
    and  \cite{IYER2014957} for remotely related works.




We organize our paper as follows: In Section \ref{MicM}, we introduce our microscopic geometry along with the microscopic model we have in mind. 
In Section \ref{assumption}, we describe the assumptions that we rely on in the mathematical analysis of  our upscaled problem.  We apply the method of two-scale asymptotic expansions with drift to our microscopic problem in the bulk of Section \ref{upscalemodel}. Here we  derive as well the structure of the upscaled equations and of the effective transport (dispersion) tensor. 
Section \ref{existence} contains our discussion on  the structural properties of the dispersion tensor and the mathematical analysis exploring the solvability of the upscaled problem. 
We close the work with a list of conclusions and a short outlook for further related research; see  Section \ref{conclusion} for details.

\par 
\section{Microscopic model}\label{MicM}
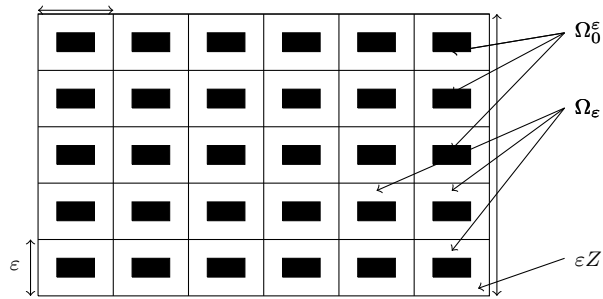
\begin{figure}[b]
		\begin{center}
			\begin{tikzpicture}
			\draw (0,0)  to (6,0) to (6,3.75) to (0,3.75) to (0,0);
			\draw (1,0) to (1,3.75)  ;
			\draw (2,0) to (2,3.75)  ;
			\draw (3,0) to (3,3.75)  ;
			\draw (4,0) to (4,3.75)  ;
			\draw (5,0) to (5,3.75)  ;
			\draw (0,.75) to (6,.75);
			\draw(0,1.5) to (6,1.5);
			\draw(0,2.25) to (6,2.25);
			\draw(0,3) to (6,3);
			\draw(0,3.75) to (6,3.75);
			\draw [fill](0.25,0.25) to (0.75,0.25) to(.75,.50) to (0.25,.50) to (.25,.25);
			\draw [fill](1.25,0.25) to (1.75,0.25) to(1.75,.50) to (1.25,.50) to (1.25,.25);
			\draw [fill](2.25,0.25) to (2.75,0.25) to(2.75,.50) to (2.25,.50) to (2.25,.25);
			\draw [fill](3.25,0.25) to (3.75,0.25) to(3.75,.50) to (3.25,.50) to (3.25,.25);
			\draw [fill](4.25,0.25) to (4.75,0.25) to(4.75,.50) to (4.25,.50) to (4.25,.25);
			\draw [fill](5.25,0.25) to (5.75,0.25) to(5.75,.50) to (5.25,.50) to (5.25,.25);
			
			\draw [fill](0.25,1.0) to (0.75,1.0) to(.75,1.25) to (0.25,1.25) to (.25,1.0);
			\draw [fill](1.25,1.) to (1.75,1.) to(1.75,1.25) to (1.25,1.25) to (1.25,1.0);
			\draw [fill](2.25,1.) to (2.75,1.) to(2.75,1.25) to (2.25,1.25) to (2.25,1.);
			\draw [fill](3.25,1.) to (3.75,1.) to(3.75,1.25) to (3.25,1.25) to (3.25,1.);
			\draw [fill](4.25,1.) to (4.75,1.) to(4.75,1.25) to (4.25,1.25) to (4.25,1.);
			\draw [fill](5.25,1.) to (5.75,1.) to(5.75,1.25) to (5.25,1.25) to (5.25,1.);
			
			\draw [fill](0.25,1.75) to (0.75,1.75) to(.75,2) to (0.25,2) to (.25,1.75);
			\draw [fill](1.25,1.75) to (1.75,1.75) to(1.75,2) to (1.25,2) to (1.25,1.75);
			\draw [fill](2.25,1.75) to (2.75,1.75) to(2.75,2) to (2.25,2) to (2.25,1.75);
			\draw [fill](3.25,1.75) to (3.75,1.75) to(3.75,2) to (3.25,2) to (3.25,1.75);
			\draw [fill](4.25,1.75) to (4.75,1.75) to(4.75,2) to (4.25,2) to (4.25,1.75);
			\draw [fill](5.25,1.75) to (5.75,1.75) to(5.75,2) to (5.25,2) to (5.25,1.75);
			
				\draw [fill](0.25,2.5) to (0.75,2.5) to(.75,2.75) to (0.25,2.75) to (.25,2.5);
			\draw [fill](1.25,2.5) to (1.75,2.5) to(1.75,2.75) to (1.25,2.75) to (1.25,2.5);
			\draw [fill](2.25,2.5) to (2.75,2.5) to(2.75,2.75) to (2.25,2.75) to (2.25,2.5);
			\draw [fill](3.25,2.5) to (3.75,2.5) to(3.75,2.75) to (3.25,2.75) to (3.25,2.5);
			\draw [fill](4.25,2.5) to (4.75,2.5) to(4.75,2.75) to (4.25,2.75) to (4.25,2.5);
			\draw [fill](5.25,2.5) to (5.75,2.5) to(5.75,2.75) to (5.25,2.75) to (5.25,2.5);
			
			\draw [fill](0.25,3.25) to (0.75,3.25) to(.75,3.5) to (0.25,3.5) to (.25,3.25);
			\draw [fill](1.25,3.25) to (1.75,3.25) to(1.75,3.5) to (1.25,3.5) to (1.25,3.25);
			\draw [fill](2.25,3.25) to (2.75,3.25) to(2.75,3.5) to (2.25,3.5) to (2.25,3.25);
			\draw [fill](3.25,3.25) to (3.75,3.25) to(3.75,3.5) to (3.25,3.5) to (3.25,3.25);
			\draw [fill](4.25,3.25) to (4.75,3.25) to(4.75,3.5) to (4.25,3.5) to (4.25,3.25);
			\draw [fill](5.25,3.25) to (5.75,3.25) to(5.75,3.5) to (5.25,3.5) to (5.25,3.25);

        	\draw[<->] (-.1,0) to (-.1,.75);
         	\draw (-.1,.4) node[anchor=east] {{\scriptsize $\varepsilon$}};
	      	\draw[<->] (0,3.80) to (1,3.8);
         	\draw [<->] (6.1,0) to (6.1,3.75);
         	\draw[->](7,3.5)node[anchor=west] {{\scriptsize $\Omega_{0}^{\varepsilon}$}} to (5.5,3.25);
         	\draw[->](7,3.5) to (5.5,2.70);
         	\draw[->](7,3.5)node[anchor=west] {{\scriptsize $\Omega_{0}^{\varepsilon}$}} to (5.5,3.25);
         	\draw[->](7,3.5) to (5.5,1.95);
         	\draw[->](7,0.5)node[anchor=west] {{\scriptsize $\es Z$}} to (5.85,0.1);
         		\draw[->](7,2.5)node[anchor=west] {{\scriptsize $\Omega_{\varepsilon}$}} to (5.5,1.4);
         		\draw[->](7,2.5)node[anchor=west] {{\scriptsize $\Omega_{\varepsilon}$}} to (4.5,1.4);
         		\draw[->](7,2.5)node[anchor=west] {{\scriptsize $\Omega_{\varepsilon}$}} to (5.5,.6);
			\end{tikzpicture}
			\caption{Schematic representation of the geometry corresponding to the microscopic model.}
			\label{fig2}
		\end{center}
	\end{figure}

Let  $Y\subset \mathbb{R}^2$ be a unit square in $\mathbb{R}^2$. We define the standard cell $Z$ as $Y$ having as inclusion an impenetrable compact object $Z_0$ called obstacle that is placed inside $Y$ ($i.e.$ $Z=Y\backslash Z_0$). We assume $\partial Z_0$ has Lipschitz boundary and $\partial Y\cap \partial Z_0=\emptyset .$ We consider that $\partial Z_0 $ has two parts, namely $\Gamma_D$ and $\Gamma_N$ (i.e. $\partial Z_0=\Gamma_D \cup \Gamma_N$ and $\Gamma_D \cap \Gamma_N=\emptyset$) with $|\Gamma_D|\geq 0,|\Gamma_N|>0$. We define the pore skeleton to be 
\begin{equation*}
    \Omega_{0}^\es:=\left\{\bigcup_{(k_1,k_2)\in \mathbb{N}\times \mathbb{N}} \{\es(Z_0+\Sigma_{i=1}^2 k_i e_i)\}\right\},
\end{equation*}
where $ \es>0$ and  $\{e_1,e_2\}$ is the orthonormal  basis of $\mathbb{R}^2$. We define the pore space and its internal boundaries as
\begin{equation*}
    \Omega_\es:=\mathbb{R}^2-\Omega_0^\es ,
\end{equation*}
\begin{equation*}
    \Gamma_{N}^\es:= \left\{\bigcup_{(k_1,k_2)\in \mathbb{N}\times \mathbb{N}} \{\es(\Gamma_N+\Sigma_{i=1}^2 k_i e_i)\}\right\}
\end{equation*}
and
\begin{equation*}
    \Gamma_{D}^\es:= \left\{\bigcup_{(k_1,k_2)\in \mathbb{N}\times \mathbb{N}} \{\es(\Gamma_D+\Sigma_{i=1}^2 k_i e_i)\}\right\},
\end{equation*}
respectively. We denote $n_\es, n_y$ as unit normal vector on $ \Gamma_{N}^\es , \Gamma_{N}$ respectively and directed outward with respect to $\Omega_\es$.
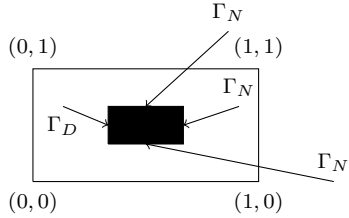
\begin{figure}[ht]
	\begin{center}
		\begin{tikzpicture}
		\draw (0,0) node [anchor=north] {{\scriptsize $(0,0)$}} to (3,0) node [anchor=north] {{\scriptsize $(1,0)$}}  to (3,1.5)node [anchor=south] {{\scriptsize $(1,1)$}} to (0,1.5)node [anchor=south] {{\scriptsize $(0,1)$}} to (0,0);
		\draw [fill] (1,.5) to (2,.5) to (2,1) to (1,1) to (1,.5);
		\draw[<-](1,0.75) to (0.4,1) node[anchor=north] {{\scriptsize $\Gamma_{D}$}};
		\draw[<-](2.0,0.75) to (2.74,1) node[anchor=south] {{\scriptsize $\Gamma_{N}$}};
			\draw[<-](1.5,1.0) to (2.6,2) node[anchor=south] {{\scriptsize $\Gamma_{N}$}};
			\draw[<-](1.5,0.5) to (4.0,-0.0) node[anchor=south] {{\scriptsize $\Gamma_{N}$}};
		\end{tikzpicture}
		\caption{Standard cell $Z$ exhibiting a rectangular obstacle $Y_0$ placed in the center.  Here $\Gamma_N$ and $\Gamma_D$ are chosen arbitrarily; see Remark \ref{R1} for details.}
		\label{scell}
	\end{center}
\end{figure}
\\
  \par  We consider the following reaction-diffusion-convection problem
\begin{align}
    \frac{\partial u^{\varepsilon}}{\partial t} +\mathrm{div}(-D^{\varepsilon}\nabla u^{\varepsilon}+ \frac{1}{\es}B^{\varepsilon}P(u^{\varepsilon}))&=f^{\varepsilon}  \,\,&\mbox{on}& \,\,&  \Omega_{\varepsilon} \times (0,T),\label{meq}\\
    (-D^{\varepsilon}\nabla u^{\varepsilon}+ \frac{1}{\es}B^{\varepsilon}P(u^{\varepsilon}))\cdot n_{\es}&=\es g_N^\es &\mbox{on}&   & \,\Gamma_N^\es \times (0,T),\label{meqbcn}\\
    u^\es&=\es^\gamma g_D^\es &\mbox{on}&   &\, \Gamma_D^\es \times (0,T)\label{meqbcd},\\
     u^\es(0)&=g &\mbox{in}&   &\, \overline{\Omega}_\es\label{meqic},
\end{align}
where $f^\es:\Omega_\es\rightarrow \mathbb{R}$, $g_N^\es:\Gamma_N^\es\rightarrow\mathbb{R}$ and $ g_D^\es:\Gamma_D^\es\rightarrow\mathbb{R}$ are given functions, $\gamma>2$, $D^\es(x_1,x_2):=D(x_1/\es, x_2/\es)$ for $(x_1,x_2)\in \Omega_\es$, where $D$ is a $2\times 2$ matrix with positive entries and $Z$--periodic defined in the standard unit cell $Z$, $B^\es(x_1,x_2):=B(x_1/\es, x_2/\es)$ where $B$ is a $2\times 1$ vector with positive entries and $Z$--periodic. 
 What concerns the nonlinear drift $P(\cdot):\mathbb{R}\rightarrow\mathbb{R}$, we consider two cases. To derive the upscaled equation in Section \ref{upscalemodel}, we take $P(\cdot)$ in the form \begin{equation}\label{p3}
    P(r):=	a_{0}+a_{1}r+\cdots+a_{m}r^{m},
\end{equation}
where $a_{k}\in \mathbb{R} \mbox{ for }k \in \mathbb{N}$. However, the well--posedness analysis of the corresponding results is much harder to reach. So, from Section \ref{existence} we use the special case of \eqref{p3}, which is
\begin{equation}\label{p}
    P(r):=r(1-r).
\end{equation}
The particular structure of the drift shown in  \eqref{p} is derived as  a mean-field limit for a totally asymmetric simple exclusion process (TASEP) on a lattice; see \cite{CIRILLO2016436} for details. 


\begin{remark}\label{R1}
Note that the structure of $\Gamma_N$ and $\Gamma_D$ need not be same as shown in the Fig. \ref{scell}. But should satisfy the conditions $\Gamma_N\cup \Gamma_D=\partial Z_0, \Gamma_N\cap \Gamma_D=\emptyset, |\Gamma_N|>0$ and $|\Gamma_D|\geq0$.
\end{remark}
\section{Assumptions}\label{assumption}
We consider the following restrictions on data and model parameters. We summarize them in the assumptions \ref{A1}--\ref{Af}, viz.
 \begin{enumerate}[label=({A}{{\arabic*}})]
 
 \item 
 For all $\eta\in \mathbb{R}^2$ there exists $\theta>0$  such that 
 \begin{equation*}\label{el}
 \theta \|\eta \|^{2} \leq
 \eta^{t}D\eta .
 \end{equation*}
 \label{A1}
\item 
$B:C^1_{\#}(Z) \rightarrow\mathbb{R}^2$ satisfies
\begin{equation*}\label{}
\begin{cases}
 	 \mathrm{div}B=0\hspace{.4cm}\mbox{in} \hspace{.4cm}  (0,T)\times Z\\
 	  B\cdot n_y =0 \hspace{.4cm}\mbox{on} \hspace{.4cm}  (0,T)\times \Gamma_N\label{boundaryb}
 	 \end{cases}
 	 \end{equation*}.
 	 \label{A2}
  \item For $(t,x)\in (0,T)\times \Omega_\es$,$f^\es(t,x):=f(t,\frac{x}{\es})$  such that $f \in L^\infty(0,T;L_\# ^2(Z));$
 
 \label{A3}
\item For $(t,x)\in (0,T)\times \Gamma_N^\es$, $g_N^\es(t,x):=g_N(t,\frac{x}{\es})$ such that $g_N\in L^\infty(0,T;L_\#^2(\Gamma_N)$ and \\
for $(t,x)\in (0,T)\times \Gamma_D^\es$, $g_D^\es(t,x):=g_D(t,\frac{x}{\es})$ such that $g_D\in L^2(0,T;L^2(\Gamma_D));$
\label{A4}
\item $g:\mathbb{R}^2\rightarrow \mathbb{R}^+\cup \{0\}$ such that
\begin{equation*}
    g\in L^\infty (\mathbb{R})\cap L^2(\mathbb{R}^2);
\end{equation*}\label{A5}
\item The inequality
\begin{equation*}
   \int_Z f\, dy-\int_{\Gamma_N}g_N d\sigma_y\geq 0
\end{equation*}
holds.\label{Af}
 \end{enumerate}
 A few comments about these assumptions are in order: Assumptions (A1), (A2), and (A6) have a clear physical justification, while (A3) and (A4) are of technical nature. 
 In particular, \ref{A2} requires the incompressibility of the drift and also mimics our expectation that particles are unable to penetrate the imposed obstacles through $\Gamma_N$ (at least not when they are travelling along the normal). 
In \ref{A3} and \ref{A4}, we considered that the functions $f^\es,g_N^\es,$ and $g_D^\es$ are depending only on the variables $t$ and $\frac{x}{\es}$, even though it is possible to consider the respective functions to be depending on the triplet $(t,x,\frac{x}{\es})$. In such case,  additional regularity is needed with respect to the second variable.

\section{Derivation of the upscaled model}\label{upscalemodel}

\subsection{Formal two-scale asymptotic expansions with drift}
To  upscale  the microscopic problem \eqref{meq}--\eqref{meqic}, we make use of the  method of two-scale asymptotic expansions with drift introduced in \cite{ALLAIRE20102292}, which later on turned into a rigorous tool  in \cite{hutridurga}  by means of the concept two-scale convergence with drift promoted in \cite{marusik2005}. 
We start with stating a Lemma needed to handle the solvability of one of the many  auxiliary problems arising in the proposed asymptotic expansion procedure. 

\begin{lemma}\label{lemma1}
Let $f_1\in L^2(Z), g_1\in L^2(\Gamma_N)$ be given functions. Then under the assumption \ref{A1}--\ref{A4}, the boundary value problem
\begin{align}
     -\nabla_y\cdot\left( D \nabla_y v\right)+\nabla_y \cdot \left(B P'(u_0)v\right)&=f_1 \hspace{.4cm}&\mbox{on} \hspace{.4cm} & (0,T)\times Z\label{1cell}\\
    \left( -D\nabla_y v+BP'(u_0)v\right)\cdot n_y &= g_1 \hspace{.4cm}&\mbox{on} \hspace{.4cm} & (0,T)\times \Gamma_N\label{1cellbn}\\
    v&=0\hspace{.4cm}&\mbox{on} \hspace{.4cm} & (0,T)\times \Gamma_D\label{1cellbd}\\
    v\,\, &\mbox{is Z-periodic},&&
\end{align}
 has a unique solution  $v\in H^1_{\#}(Z)$ if and only if the compatibility condition
\begin{equation}\label{compat}
    \int_Z f_1\, dy=\int_{\partial Z} g_1\,d\sigma_y
\end{equation}
is satisfied.
\end{lemma}
\begin{proof}
The proof of this statement follows by standard argument involving the classical Fredholm alternative, for details we refer the reader to Lemma 1.3.21 of \cite{allaireshape2002}.
\end{proof}

\par As starting point of the upscaling work, we assume $u^\es$ satisfies the following infinite series expansion 
\begin{equation}\label{ansatz}
    u^\es (t,x)=\sum_{k=0}^\infty \es ^k u_k\left(t,x-\frac{B^*t}{\es},\frac{x}{\es}\right),
\end{equation}
where the function $u_k(t,x,y)$ is $Z$-periodic in the variable $y\in Z$ for any $k\in \mathbb{N}\cup \{0\}$, the vector $B^*$ is the effective drift, whose value will be identified  at a later stage. Alternatively, one could also the general form 
\begin{equation*}\label{ansatz2}
    u^\es (t,x)=\sum_{k=0}^\infty \es ^k u_k\left(t,x-\frac{B^*(t)}{\es},\frac{x}{\es}\right).
\end{equation*}
In this context, such choice leads to $B^*(t)=B^*t$. 
\newline
We use the transformation $X=x-\frac{B^*t}{\es}$ and the chain rule for the differentiation, we obtain the following identities:
\begin{align}
    \frac{\partial }{\partial t}\left(u^\es\left(t,x-\frac{B^*t}{\es},\frac{x}{\es}\right) \right)&=\frac{\partial }{\partial t}u^\es\left(t,x-\frac{B^*t}{\es},\frac{x}{\es}\right) -\frac{B^*}{\es}\nabla_X u^\es\left(t,x-\frac{B^*t}{\es},\frac{x}{\es}\right) \label{timeder}\\
    \nabla \left(u^\es\left(t,x-\frac{B^*t}{\es},\frac{x}{\es}\right) \right)&=\nabla_X u^\es\left(t,x-\frac{B^*t}{\es},\frac{x}{\es}\right)+\frac{1}{\es} \nabla_y u^\es\left(t,x-\frac{B^*t}{\es},\frac{x}{\es}\right)
\end{align}
\begin{multline}
    \nabla\cdot D^\es \nabla \left(u^\es\left(t,x-\frac{B^*t}{\es},\frac{x}{\es}\right) \right)=\nabla_X\cdot D^\es \nabla_X u^\es\left(t,x-\frac{B^*t}{\es},\frac{x}{\es}\right)
    +\frac{1}{\es} \nabla_X\cdot D^\es \nabla_y u^\es\left(t,x-\frac{B^*t}{\es},\frac{x}{\es}\right)\\
   +\frac{1}{\es}\nabla_y\cdot D^\es \nabla_X u^\es\left(t,x-\frac{B^*t}{\es},\frac{x}{\es}\right)
   +\frac{1}{\es^2}\nabla_y\cdot D^\es \nabla_y u^\es\left(t,x-\frac{B^*t}{\es},\frac{x}{\es}\right)\label{diffexp}.
\end{multline}
For convenience, we denote  $\nabla_X$ as $\nabla_x$.
Using  \eqref{ansatz} and on the fact that the nonlinear term $P(\cdot)$ arising in \eqref{meq} satisfies  $P\in C^\infty(\mathbb{R})$, we can write the following power series expansion of $P$ around some $u_0$:
\begin{equation}\label{taylor}
P(u^\es)=P(u_0)+P'(u_0)(u^\es-u_0)+\frac{1}{2}P''(u_0)(u^\es-u_0)^{2}+\sum_{i=3}^\infty P^n(u_{0})(u^\es-u_0)^{n}.
\end{equation}
Later on, we will treat $u_0$ as the limit of $u_\es$ when $\es \rightarrow 0$.
 Now, substituting \eqref{ansatz} into \eqref{taylor}, we get
 \begin{equation}\label{pu}
 \begin{aligned}
P(u^\es)&=P(u_0)+\es P'(u_0)u_1+\es^2 (P'(u_0) u_2+ \frac{1}{2}P''(u_0)u_1^{2}+\sum_{i=3}^\infty P^n(u_{0})(u^\es-u_0)^{n}\\
&=P(u_0)+\es P'(u_0)u_1+\es^2 (P'(u_0) u_2+ \frac{1}{2}P''(u_0)u_1^{2})+O(\es^3).
\end{aligned}
\end{equation}
Using \eqref{pu} and the chain rule of differentiation, we have
\begin{equation}\label{driftexp}
\begin{aligned}
    \nabla \cdot \left(B^\es P(u_\es)\right) = &\frac{1}{\es}\nabla_y\cdot\left(  B^\es P(u_0)\right)+\es^0\left(\nabla_x \cdot B^\es P(u_0)+\nabla_y \cdot \left(B^\es P'(u_0)u_1\right)\right)\\
    &\hspace{1cm}+\es\left( \nabla_x \cdot\left( B^\es P'(u_0)u_1\right)+\nabla_y \cdot \left(B^\es P''(u_0)u_1^2\right)\right)+O(\es^3).
    \end{aligned}
\end{equation}
 Substituting \eqref{ansatz} into \eqref{meq}-\eqref{meqbcd}, using \eqref{timeder}--\eqref{diffexp} and \eqref{driftexp}, and finally, collecting $\es^{-2}$ order terms from \eqref{meq}, $\es^{-1}$ order terms from \eqref{meqbcn} and also $\es^{0}$ order terms from \eqref{meqbcd}, we obtain:
\begin{align}
   - \nabla_y\cdot D(y) \nabla_y u_0+\nabla_y \cdot \left(B(y)P(u_0)\right)&=0 \hspace{1cm}&\mbox{on \hspace{1cm}} & (0,T)\times Z\label{xe1}\\
   \left( -D(y)\nabla_y u_0+B(y)u_0 \right)\cdot n_y&=0\hspace{1cm}&\mbox{on \hspace{1cm}} & (0,T)\times \Gamma_N \label{xe2}\\
    u_0&=0  \hspace{1cm}&\mbox{on \hspace{1cm}} & (0,T)\times \Gamma_D.
\end{align}
Note that \eqref{xe2} implies $\left( -D(y)\nabla_y u_0 \right)\cdot n_y=0$ on $(0,T)\times \Gamma_N$. We now show that $u_0(t,\cdot)$ depends on $x$, but it does not depend on $y$. This is a crucial step in the derivation of the upscaled equations presented here. Using a straightforward integration by parts, \ref{A2}, the periodicity of $B(\cdot)$, and the periodicity of $u_0(t,x,\cdot)$, we have 
\begin{equation}
\begin{aligned}\label{bu0}
    \int_Z B(y) u_0 ^n \cdot \nabla_y u_0\, dy&= \frac{1}{n+1} \int_Z B(y) \cdot \nabla_y (u_0^{n+1})\, dy\\
    &=\frac{1}{n+1}\left(-\int_Z (\nabla_y \cdot B(y))\,  u_0^{n+1} \, dy+ \int_{\partial Z}u_0^{n+1}B(y)   \cdot n_y d\sigma_y\right)\\
    &=0.
\end{aligned}
\end{equation}
Relation \eqref{bu0} together with the structure of $P(\cdot)$ allow us to obtain
\begin{equation}\label{bu01}
     \int_Z B(y) P(u_0) \cdot \nabla_y u_0\, dy=0.
\end{equation}
Combining \ref{A1}, \ref{A2},\eqref{xe1}  the integration by parts and the $Z-$periodicity property of $B(y), D(y),$ and of $u_0(t,x,y)$ $(y\in Z)$, with \eqref{bu01} leads to
\begin{equation}\label{yind}
    \begin{aligned}
        \theta\int_Z  |\nabla_y u_0|^{2}\, dy &\leq \int_Z \nabla_y u_0\cdot D(y) \nabla_y u_0\, dy \\
        &=-\int_Z \nabla_y\cdot (D(y) \nabla_y u_0) u_0\, dy+\int_{\partial Z}u_0D(y) \nabla_y u_0 \cdot n_y d\sigma_y\\
        &=-\int_Z \nabla_y\cdot (B(y) P(u_0)) u_0\, dy+\int_{\partial Z}u_0D(y) \nabla_y u_0 \cdot n_y d\sigma_y\\
        &=\int_ZB(y) P(u_0) \nabla_y u_0\, dy-\int_{\partial Z}u_0B(y) P(u_0) \cdot n_y d\sigma_y+\int_{\partial Z}u_0D(y) \nabla_y u_0 \cdot n_y\, d\sigma_y\\
        &=0.
    \end{aligned}
\end{equation}
We can rely on \eqref{yind}, to conclude that $u_0$ is independent on $y$, i.e.

\begin{equation}\label{xind}
u_0(t,x,y)=u_0(t,x)\hspace{2cm}\mbox{for} \hspace{0.5cm}x\in \Omega_\es\mbox{\,\,and\,\,} \, t\in (0,T).
\end{equation}

\par Collecting now the $\es^{-1}$ order terms from \eqref{meq}, the $\es^0$ order terms from \eqref{meqbcn}, the $\es^1$ order terms from \eqref{meqbcd} and using \eqref{xind}, we get
    \begin{multline}\label{e2}
         -\nabla_y\cdot D(y) \nabla_x u_0-\nabla_y\cdot D(y) \nabla_y u_1-B^*\nabla_x u_0\\
         +\nabla_x \cdot \left(B(y)P(u_0)\right)+\nabla_y \cdot \left(B(y)P'(u_0)u_1\right)=0\hspace{.4cm}\mbox{on} \hspace{.4cm}  (0,T)\times Z,
     \end{multline}
    \begin{align}
    \left(-D(y)\nabla_x u_0-D(y)\nabla_y u_1+B(y)P'(u_0)u_1\right)\cdot n_y&=0 \hspace{.4cm}&\mbox{on} \hspace{.4cm} & (0,T)\times \Gamma_N,\label{e2n}\\
    u_1&=0 \hspace{.4cm}&\mbox{on} \hspace{.4cm} & (0,T)\times \Gamma_D\label{e2d}.
     \end{align}
By \ref{A2}, we ensure 
\begin{equation}\label{ch1}
    \nabla_x \cdot\left( BP(u_0)\right)=B(y)\cdot P'(u_0)\nabla_x u_0
\end{equation}
 To obtain the value of $B^*$ and some information on the oscillating  structure of $u_1$, we introduce a cell problem related to \eqref{e2}-\eqref{e2d}. The structure of problem \eqref{e2}--\eqref{e2d} and jointly with equation \eqref{ch1} allow us to look for $u_1$ in the form
\begin{equation}\label{wu}
    u_1(t,x,y)=W(y)\cdot \nabla_xu_0+\Tilde{u}_0(x),
\end{equation}
where $\Tilde{u}_0$ is some given function and  $W(y):=(w_1(y),w_2(y))$. The components  $w_i$ $(i\in\{1,2\})$ satisfy the following \emph{cell problems:}
\begin{align}
    -\nabla_y \cdot (D(y) \nabla_y w_i)+\nabla_y \cdot \left(B(y) P'(u_0)w_i\right)&=\nabla_y \cdot (D(y) e_i)+B^*\cdot e_i-B(y)&\cdot P'(u_0) e_i&\nonumber\\
   & \hspace{.4cm}&\mbox{on} \hspace{.4cm} & (0,T)\times Z,\label{cell}\\
    \left( -D(y)\nabla_y w_i+BP'(u_0)w_i\right)\cdot n_y &= \left( -D(y)e_i\right)\cdot n_y \hspace{.4cm}&\mbox{on} \hspace{.4cm} & (0,T)\times \Gamma_N,\label{cellbn}\\
    w_i&=0\hspace{.4cm}&\mbox{on} \hspace{.4cm} & (0,T)\times \Gamma_D,\label{cellbd}\\
    w_i\, &\mbox{ is $Z$--periodic}.\label{cellper}&&
\end{align}
Let $B_R(0)\subset\mathbb{R}^2$ be a ball of radius $R$ centered at orgin. Multiplying \eqref{cell} and \eqref{cellbn} by $ \chi_{B_R(0)}$, integrating the corresponding result  over $(0,T)\times B_R(0)$ and taking $R\rightarrow\infty$, we get
\begin{multline}\label{cell2}
    -\nabla_y \cdot D(y) \nabla_y w_i+\lim_{R\rightarrow\infty}\frac{\int_0^T\int_{B_R(0)} P'(u_0)dxdt}{T|B_R(0)|}\nabla_y \cdot B(y) w_i\\=\nabla_y \cdot D(y) e_i+B^*\cdot e_i-\lim_{R\rightarrow\infty}\frac{\int_0^T\int_{B_R(0)} P'(u_0)dxdt}{T|B_R(0)|}B(y)\cdot e_i \hspace{.4cm}\mbox{on} \hspace{.4cm}  (0,T)\times Z,
\end{multline}
\begin{align}
    \left( -D(y)\nabla_y w_i+\lim_{R\rightarrow\infty}\frac{\int_0^T\int_{B_R(0)} P'(u_0)dxdt}{T|B_R(0)|}B w_i\right)\cdot n_y &= \left( -D(y)e_i\right)\cdot n_y \hspace{.4cm}&\mbox{on} \hspace{.4cm} & (0,T)\times \Gamma_N,\label{cellbn2}\\
    w_i&=0\hspace{.4cm}&\mbox{on} \hspace{.4cm} & (0,T)\times \Gamma_D,\label{cellbd2}\\
    w_i\, &\mbox{ is $Z$--periodic}.&&
\end{align}
We define the quantity $\lim_{R\rightarrow\infty}$ average of $P'(\cdot) $ over $(0,T)\times B_{R}(0)$ as 
\begin{equation}
    \mathfrak{A}(u_0):=\lim_{R\rightarrow\infty}\frac{\int_0^T\int_{B_R(0)} P'(u_0)dxdt}{T|B_R(0)|}.
\end{equation} Using the compatibility condition \eqref{compat} stated in Lemma \ref{lemma1},  we deduce that there exists a unique solution to the problem \eqref{cell}-\eqref{cellbd}, if and only if 
\begin{equation}\label{core1}
    \int_Z \nabla_y\cdot D e_i+B^*\cdot e_i-\mathfrak{A}(u_0)B(y)\cdot e_i dy=\int_{\partial Z}\left( -D(y)e_i\right)\cdot n_y d\sigma_y.
    \end{equation} 
Equation \eqref{core1} allows us to fix the entries of the vector $B^*$ indicated in \eqref{ansatz}. Namely, we set
\begin{equation}\label{bstar}
    B^*\cdot e_i:=\frac{\int_Z -\nabla_y\cdot D e_i+\mathfrak{A}(u_0)B(y)\cdot  e_i dy +\int_{\partial Z}\left( -D(y)e_i\right)\cdot n_y d\sigma_y}{|Z|},
\end{equation}
where $|Z|$ denotes the volume of the cell $Z$.
\par Collecting the $\es^{0}$ order terms from \eqref{meq}, the $\es^1$ order terms from \eqref{meqbcn}, the $\es^2$ order terms from \eqref{meqbcd} and using \eqref{xind}, we get
\begin{multline}\label{u2e}
   -\nabla_y\cdot \left(D(y)\nabla_y u_2+ (BP'(u_0)u_2)\right)= -\partial_t u_0 +\nabla_x\cdot\left( D(y)\nabla_x u_0+ D(y)\nabla_y u_1\right)+\nabla_y\cdot D(y)\nabla_x u_1\\
    +B^* \cdot \nabla_x u_1-\nabla_x \cdot \left(B(y)P'(u_0)u_1\right)-\frac{1}{2} \nabla_y\cdot \left(B(y)P^{''}(u_0) u_1^2\right)\\
    +f \hspace{.4cm}\mbox{on} \hspace{.4cm}  (0,T)\times Z,
\end{multline}

\begin{equation}\label{u2ebn}
    \left( -D(y)\nabla_y u_2  +BP'(u_0)u_2\right)\cdot n_y=\left( D(y)\nabla_x u_1-B(y)P^{''}(u_0)u_1^2  \right)\cdot n_y+g_N\hspace{.4cm}\mbox{on} \hspace{.4cm}  (0,T)\times \Gamma_N,
\end{equation}

\begin{equation}\label{u2ebd}
    u_2=0\hspace{.4cm}\mbox{on} \hspace{.4cm}  (0,T)\times \Gamma_D,
\end{equation}
where $f$ and $g_N$ are restriction of $f^\es$ and $g_N^\es$ on $Z$ and $\Gamma_N$ respectively.
 \par Referring again to Lemma \ref{lemma1} as applied to the problem  \eqref{u2e}--\eqref{u2ebd}, to hold the existence of   $u_2$,  the following compatibility condition has to be fulfilled:
\begin{multline*}\label{}
    \int_Z \left(-\partial_t u_0 +\nabla_x\cdot\left( D(y)\nabla_x u_0+ D(y)\nabla_y u_1\right)+\nabla_y\cdot (D(y)\nabla_x u_1)\right.\\
   \left. +B^*\cdot \nabla_x u_1-\nabla_x \cdot\left( B(y)P'(u_0)u_1\right)-\frac{1}{2} \nabla_y\cdot \left(B(y)P^{''}(u_0) u_1^2 \right)   +f \right) \, dy\\
    =\int_{\Gamma_N}\left(\left( D(y)\nabla_x u_1-\frac{1}{2}BP^{''}(u_0)u_1^2\right)\cdot n_y +g_N\, \right)d\sigma_y.
\end{multline*}
Consequently, we obtain
\begin{multline}\label{e32}
    -|Z|\partial_t u_0 +\nabla_x\cdot  \int_Z D(y)\nabla_x u_0\, dy+\nabla_x\cdot \int_Z D(y)\nabla_y u_1\, dy
    +B^* \cdot \nabla_x\int_Z u_1\, dy\\
    -\nabla_x \cdot \int_Z B(y)P'(u_0)u_1\, dy 
    =-\int_Z f\, dy+\int_{\Gamma_N}g_N \, d\sigma_y.
\end{multline}
Now, by using \eqref{wu}, we replace $u_1$ in terms of $W$  and rearrange the terms of \eqref{e32}. This yields

\begin{equation*}\label{}
    \partial_t u_0 +\mathrm{div}( -D^*(u_0,W)\nabla_x  u_0)=\frac{1}{|Z|}\int_Z f\, dy + \frac{-1}{|Z|}\int_{\Gamma_N}g_N \, d\sigma_y.
\end{equation*}
 where the obtained effective transport tensor takes the explicit form:
 \begin{multline*}\label{}
     D^*(u_0,W):=\frac{1}{|Z|}\int_Z D(y)\left(I+\begin{bmatrix}
\frac{\partial w_1}{\partial y_1} & \frac{\partial w_2}{\partial y_1} \\
\frac{\partial w_1}{\partial y_2} & \frac{\partial w_2}{\partial y_2}
\end{bmatrix}\right)\, dy
\\
+\frac{1}{|Z|}B^*(u_0)\int_Z W(y)^t\, dy-\frac{P'(u_0)}{|Z|}\int_ZB(y) W(y)^t\, dy
 \end{multline*}
 Its worth noticing that the first term of $D^*(\cdot)$ refers to an averaged diffusion contribution, while the other two accounts for averaged drift effects. If $P(\cdot)$ is linear, then one recovers the results from  \cite{ALLAIRE20102292}.
 \subsection{Summary of the upscaled model equations}\label{upscaled}
In this section we summarize the obtained upscaled equations.
\begin{tcolorbox}
Find  $(u_0,W)$ satisfying the following system of equations with $ W=(w_1,w_2)$
 \begin{align}
    \partial_t u_0 +\mathrm{div}( -D^*(u_0,W)\nabla_x  u_0)&=\frac{1}{|Z|}\int_Z f\, dy + \frac{-1}{|Z|}\int_{\Gamma_N}g_N \, d\sigma_y &\mbox{on}& \hspace{.4cm}  (0,T)\times \mathbb{R}^2, \label{homeq1}\\
    u_0(0)&=g &\mbox{on}& \hspace{1.75cm}  \mathbb{R}^2
    \end{align}
    \end{tcolorbox}
    \begin{tcolorbox}
    \begin{align}
    -\nabla_y\cdot D(y) \nabla_y w_i+\mathfrak{A}(u_0)\nabla_y \cdot \left(B(y) w_i\right)&=\nabla_y D(y) e_i+B^* e_i-\mathfrak{A}(u_0)B(y)\cdot  e_i &\mbox{on}& \hspace{.4cm}  (0,T)\times Z,\label{cell3}\\
    \left( -D(y)\nabla_y w_i+B\mathfrak{A}(u_0)w_i\right)\cdot n_y &= \left( -D(y)e_i\right)\cdot n_y &\mbox{on}& \hspace{.4cm} (0,T)\times \Gamma_N,\label{cellbn3}\\
    w_i&=0&\mbox{on}& \hspace{.4cm}  (0,T)\times \Gamma_D,\label{cellbd3}\\
    w_i\, &\mbox{ is $Z$--periodic},&&\label{homeqf}
\end{align}
where $i\in \{1,2\}$. 
\end{tcolorbox}
\begin{tcolorbox}
The effective dispersion tensor $D^*$ is defined as
\begin{multline}\label{effectdiff}
     D^*(u_0,W)=\frac{1}{|Z|}\int_Z D(y)\left(I+\begin{bmatrix}
\frac{\partial w_1}{\partial y_1} & \frac{\partial w_2}{\partial y_1} \\
\frac{\partial w_1}{\partial y_2} & \frac{\partial w_2}{\partial y_2}
\end{bmatrix}\right)\, dy
\\
+\frac{1}{|Z|}B^*(u_0)\int_Z W(y)^t\, dy-\frac{P'(u_0)}{|Z|}\int_ZB(y) W(y)^t\, dy.
 \end{multline}
 \end{tcolorbox}
Note that this system is not only fully coupled, but it is also posed on two different spatial scales (micro and macro) where the variables $x\in\mathbb{R}^2$ and $y\in Z$ are defined.   The terminology ``\emph{effective dispersion tensor}" is taken from the porous media literature; see in particular the terminology used in the monograph \cite{bear1988dynamics} as well as \cite{vafai2015handbook}. 

We refer to the set of equations \eqref{homeq1}--\eqref{homeqf} together with (\ref{effectdiff}) as problem $P(\Omega)$.
 
 \section{Solvability of the upscaled problem \texorpdfstring{$P(\Omega)$}{TEXT}}\label{existence}
 \subsection{Structural properties of the dispersion tensor  \texorpdfstring{$D^*(\cdot)$}{TEXT}}

\begin{proposition}\label{L2}
Assume  \ref{A1}--\ref{A5}. Then  the effective dispersion tensor $D^*$ can be decomposed as $D^*=A^*+J^*$, where 
\begin{equation}\label{}
    [A^*]_{i,j}:=\frac{1}{|Z|}\int_Z D(y)(e_j+\nabla_yw_j)\cdot 
    (e_i+\nabla_y w_i)\, dy
    \end{equation}
and
\begin{align}\label{}
    [J^*]_{i,j}:=&\frac{P'(u_0)}{|Z|}\int_Z B(y)\cdot(w_ie_j-w_je_i)\, dy\\
    &+\frac{1}{|Z|}B^*(u_0)\cdot \int_Z (w_je_i-w_ie_j)-\frac{P'(u_0)}{|Z|}\int_Z w_iB(y)\cdot\nabla_y w_i\, dy,
    \end{align}
Furthermore, $A^*$ is symmetric, $J^*$ is anti-symmetric and for any $\xi \in \mathbb{R}^2$ there exist  $\alpha >0$ such that 
\begin{equation}\label{dst}
    \xi ^t D^* \xi\geq \alpha |\xi|^2,
\end{equation}
i.e., $D^*$ is uniformly positive definite.
\end{proposition}
\begin{proof}
From  \eqref{effectdiff} we obtain
\begin{equation}\label{L2e1}
     [D^*(u_0,W)]_{i,j}=\frac{1}{|Z|}\int_Z D(y)\left(e_j+\nabla_y w_j\right)\cdot e_i\, dy
+\frac{1}{|Z|}B^*(u_0)\cdot \int_Z w_j e_i\, dy\\-\frac{P'(u_0)}{|Z|}\int_ZB(y)\cdot w_j e_i\, dy.
\end{equation}
Now, we consider \eqref{cell} for $w_j$, and multiply it with $w_i$. We get
\begin{multline}\label{L2e2}
    -\nabla_y\cdot( D(y) \nabla_y w_j)w_i+\nabla_y \cdot (B(y) P'(u_0)w_j)w_i-\nabla_y \cdot (D(y) e_j)w_i\\
    -B^*\cdot e_j w_i+B(y)\cdot P'(u_0) e_j w_i=0 \hspace{.4cm}\mbox{on} \hspace{.4cm}  (0,T)\times Z.
\end{multline}
Multiplying \eqref{L2e2} by $\frac{1}{|Z|}$ and integrating the result over $Z$ yields 
\begin{multline}\label{L2e3}
    \frac{1}{|Z|}\int_Z-\nabla_y\cdot( D(y) \nabla_y w_j)w_i\,dy+\frac{1}{|Z|}\int_Z\nabla_y \cdot (B(y) P'(u_0)w_j)w_i\,dy\\
    -\frac{1}{|Z|}\int_Z\nabla_y \cdot (D(y) e_j)w_i\,dy
    -\frac{1}{|Z|}B^*\cdot\int_Z e_j w_i\,dy+\frac{1}{|Z|}\int_ZB(y)\cdot P'(u_0) e_j w_i\,dy=0.
\end{multline}
Performing the integration by parts on the first three terms of \eqref{L2e3}, using \eqref{cellbn}--\eqref{cellper} as well as employing the  $Z-$periodicity of $D$ and $B$, gives
\begin{multline}\label{L2e4}
    \frac{1}{|Z|}\int_Z \nabla_y w_i\cdot D(y) \nabla_y w_j\,dy-\frac{P'(u_0)}{|Z|}\int_Z  (B(y) w_j)\cdot \nabla_yw_i\,dy\\
    +\frac{1}{|Z|}\int_Z  (D(y) e_j)\cdot \nabla_yw_i\,dy
    -\frac{1}{|Z|}B^*\cdot\int_Z e_j w_i\,dy+\frac{P'(u_0)}{|Z|}\int_ZB(y)\cdot  e_j w_i\,dy=0.
\end{multline}
Adding \eqref{L2e4} to \eqref{L2e1}, we are lead to the decomposition:
\begin{equation}\label{L2e5}
    D^*=A^*+J^*,
\end{equation}
where the terms $A^*$ and $J^*$ are 
\begin{equation}\label{L2e6}
    [A^*]_{i,j}=\frac{1}{|Z|}\int_Z D(y)(e_j+\nabla_yw_j)\cdot 
    (e_i+\nabla_y w_i)\, dy
    \end{equation}
    and
\begin{align}\label{L2e7}
    [J^*]_{i,j}=&\frac{P'(u_0)}{|Z|}\int_Z B(y)\cdot(w_ie_j-w_je_i)\, dy\\
    &+\frac{1}{|Z|}B^*(u_0)\cdot \int_Z (w_je_i-w_ie_j)-\frac{P'(u_0)}{|Z|}\int_Z w_iB(y)\cdot\nabla_y w_i\, dy,
    \end{align}
    respectively. Observe that $A^*$ refers to effective diffusion components, while $J^*$ includes components of the effective drift.\\
     Assumption  \ref{A2}, together with the integration by parts, and with the $Z-$periodicity of $w_i, w_j$,  and of $B$, yields
    \begin{align}
        \int_Z w_jB(y)\cdot \nabla_y w_i \, dy&=-\int_Z \nabla_y\cdot (B(y)w_j)  w_i \, dy+\int_{\partial Z} (B(y)w_j w_i)\cdot n_y\, d\sigma_y\label{L2e8}\\
        &=-\int_Z  B(y)\cdot \nabla_yw_j  w_i \, dy.\label{L2e9}
    \end{align}
    Considering the right--hand side of \eqref{L2e6}, we see that $A^*$ is a symmetric matrix. Inserting \eqref{L2e9} in the last term of \eqref{L2e7}, we see that the matrix $J^*$ can be written in the form of a skew-symmetric matrix. Hence, we conclude that the effective dispersion tensor $D^*$ of the upscaled problem \eqref{homeq1} can be write as sum of a symmetric matrix $A^*$, defined as \eqref{L2e6}, and a skew-symmetric matrix $J^*$, defined as \eqref{L2e7}.
    \par To prove the uniform positivity property of $D^*$, it is enough to prove that $A^*$ is uniformly positive definite. Note that since $D^*=A^*+J^*$ and because  $J^*$ is a $2\times 2$ skew symmetric matrix, we have $\xi^t J^* \xi =0$.
    
    Using the expression \eqref{L2e6}, we have
    \begin{equation*}
        A^*\xi \cdot \xi = \int_Z D(y)(\xi +\nabla_y\sum_{i=1}^{2}w_i\xi_i)\cdot 
    (\xi +\nabla_y\sum_{i=1}^{2}w_i\xi_i)\, dy.
    \end{equation*}
    Assumption \ref{A1} together with the periodicity property of $w_1$ and $w_2$ implies
    \begin{align*}
        A^*\xi\cdot \xi &\geq \theta \int_Z |\xi+\nabla_y(\sum_{i=1}^{2}\xi_i w_i)|^2\,dy\\
        &=\theta \int_Z |\xi|^2\, dy + \theta \int_Z|(\sum_{i=1}^{2}\xi_i w_i)|^2+2\theta \int_Z \xi \cdot \nabla_y (\sum_{i=1}^{2}\xi_i w_i)\, dy\\
        &\geq \theta |Z| |\xi|^2.
    \end{align*}
    Let $\alpha:=\theta |Z|$ we get \eqref{dst}.
\end{proof}

\begin{remark}
The decomposition stated in Proposition \ref{L2} can be obtained by computing
\begin{equation*}
A^*=\frac{D^*+(D^*)^T}{2}
   \quad\text{and}\quad 
 J^*=\frac{D^*-(D^*)^T}{2}.
\end{equation*}
\end{remark}
\subsection{Weak solvability on a bounded domain }\label{bdd}
In this section, we study weak solvability of our upscaled model \eqref{homeq1}--\eqref{homeqf}  posed on a bounded smooth domain. Specifically, we prove the existence of a weak solution for this problem. Later, using  techniques from \cite{HILHORST20071118} and our existence result for a bounded domain, we investigate the solvability  of $P(\Omega)-$ the upscaled problem posed in unbounded domain. 
\par Let $L>0$ be arbitrarily fixed. Take $\Omega_L\subset\mathbb{R}^2$ a domain with $\partial\Omega_L\in C^1$ and having diameter $2L$. We define  problem $P(\Omega_L)$ as follows
\begin{align}
    \partial_t u_0 +\mathrm{div}( -D^*(u_0,W)\nabla_x  u_0)&=\Tilde{f} &\mbox{on}& \hspace{.4cm}  (0,T)\times \Omega_L, \label{Lhomeq1}\\
    u_0&=0&\mbox{on}& \hspace{.4cm}  (0,T)\times \partial\Omega_L\\
    u_0(0,x)&=g &\mbox{on}& \hspace{.4cm}  x\in \Omega_L \label{bdic}\\
    -\nabla_y\cdot D(y) \nabla_y w_i+\mathfrak{A}(u_0)\nabla_y \cdot \left(B(y) w_i\right)&=\nabla_y \cdot D(y) e_i+B^* \cdot e_i-\mathfrak{A}(u_0)B(y)\cdot  e_i &\mbox{on}& \hspace{.3cm}  (0,T)\times Z,\label{Lcell3}\\
    \left( -D(y)\nabla_y w_i+\mathfrak{A}(u_0)Bw_i\right)\cdot n_y &= \left( -D(y)e_i\right)\cdot n_y &\mbox{on}& \hspace{.4cm} (0,T)\times \Gamma_N,\label{Lcellbn3}\\
    w_i&=0&\mbox{on}& \hspace{.4cm}  (0,T)\times \Gamma_D,\label{Lcellbd3}\\
    w_i\, &\mbox{ is $Z$--periodic},&&\label{Lcellp}
\end{align}
where \begin{equation*}
    \Tilde{f}=\frac{1}{|Z|}\int_Z f\, dy + \frac{-1}{|Z|}\int_{\Gamma_N}g_N \, d\sigma_y 
\end{equation*}
$D^*, B^*, D(y), B(y), P, f$,$g_N$ and $g$ are defined as \eqref{effectdiff}, \eqref{bstar}, \ref{A1}, \ref{A2}, \eqref{p}, \ref{A3}, \ref{A4} and \ref{A5} respectively. We refer to \eqref{Lhomeq1}--\eqref{Lcellp} as problem $P(\Omega_L).$ Note that if we consider $\Omega_L$ as a ball with radius $L$, then as $L\rightarrow \infty$ the problem $P(\Omega_L)$ is supposed to approximate $P(\Omega)$. We  see this as a regular asymptotic expansion. It is the aim of this section to make this assumption rigorous. 
\par We define the weak formulation of \eqref{Lhomeq1}--\eqref{Lcellp} as follows.
\begin{definition}\label{D1}
    The pair $(u_0,W)\in L^2(0,T;H^1(\Omega_L))\times [H_\#^1(Z)]^2$ is called weak solution to  $P(\Omega_L)$, if and only if the following identities are satisfied:
    \begin{align}
    \int_{\Omega_L} \partial_t u_0 \phi \, dx+ \int_{\Omega_L}\nabla \phi\cdot D^*(u_0,W) \nabla u_0  \, dx&=  \int_{\Omega_L}\left( \frac{1}{|Z|}\int_Z f\, dy + \frac{-1}{|Z|}\int_{\Gamma_N}g_Nd\sigma_y\right)\phi \, dx\label{wfbd1}\\
    \int_{Z}\nabla_y \psi\cdot D(y)\nabla_y w_i dy+ \mathfrak{A}(u_0)\int_{Z}w_iB(y)\cdot \nabla_y \psi\,  dy&=\int_Z \left(B^*\cdot e_i-B(y)\cdot \mathfrak{A}(u_0) e_i\right)\psi \, dy\label{wfbp}
\end{align}
for all $(\phi,\psi)\in H^1(\Omega_L)\times H^1_\#(Z)$ and $a.e.$ $t\in (0,T),$ together with the initial condition
\begin{equation}\label{wfic}
    u_0(0)=g \hspace{.5cm} \mbox{on} \hspace{.5cm} \overline{\Omega}_L
\end{equation}
\end{definition}
    \begin{theorem}\label{T1}
    Assume that \ref{A1}--\ref{A5} hold true. Then  $P(\Omega_L)$ admits  a solution in the sense of Definition \ref{D1}.

    \end{theorem}
    \begin{proof}
    We prove the existence of weak solutions to $P(\Omega_L)$, i.e. to \eqref{Lhomeq1}--\eqref{Lcellp} by using a variant of the classical Schauder's fixed point theorem (see Theorem 3 in section 9.2.2 of \cite{evans2010partial}).\\
    We define a map $Q:L^2(0,T;L^2(\Omega_L))\rightarrow L^2(0,T;L^2(\Omega_L))$  such that $Q(v):=p $, where $p$ is the solution to the following problem \eqref{1Lhomeq1}--\eqref{1Lcellp}, viz.
     \begin{align}
    \partial_t p +\mathrm{div}( -D^*(v,W)\nabla_x  p)&=\Tilde{f}  &\mbox{on}& \hspace{.4cm}  (0,T)\times \Omega_L, \label{1Lhomeq1}\\
     p&=0&\mbox{on}& \hspace{.4cm}  (0,T)\times \partial \Omega_L\\
    p(0,x)&=g &\mbox{for}& \hspace{.4cm}  x\in \overline{\Omega}_L \label{1Lic}\\
    -\nabla_y D(y) \nabla_y w_i+\nabla_y \cdot B(y) \mathfrak{A}(v)w_i&=\nabla_y\cdot D(y) e_i+B^* e_i-B(y)\cdot \mathfrak{A}(v) e_i &\mbox{on}& \hspace{.4cm}  (0,T)\times Z,\label{1Lcell3}\\
    \left( -D(y)\nabla_y w_i+B\mathfrak{A}(v)w_i\right)\cdot n_y &= \left( -D(y)e_i\right)\cdot n_y &\mbox{on}& \hspace{.4cm} (0,T)\times \Gamma_N,\label{1Lcellbn3}\\
    w_i&=0&\mbox{on}& \hspace{.4cm}  (0,T)\times \Gamma_D,\label{1Lcellbd3}\\
    w_i\, &\mbox{ is $Z$--periodic},&&\label{1Lcellp}
\end{align}
where $\Tilde{f}:=\frac{1}{|Z|}\int_Z f\, dy + \frac{-1}{|Z|}\int_{\Gamma_N}g_Nd\sigma_y$.  Since the problem \eqref{1Lcell3}--\eqref{1Lcellp} is independent of $p$ and is in fact a linear elliptic equation, Lax-Milgram Lemma (see Chapter 6 of \cite{evans2010partial}) ensures the existence of a unique solution $w_i\in H^1_\# (Z)/\mathbb{R}$ (if $|\Gamma_D|=0$ or  $w_i\in H^1_\# (Z)$ if $|\Gamma_D|>0$). Using Proposition \ref{L2} and  the standard parabolic theory, we deduce that there exists an  unique weak solution lying in $L^2(0,T;H^1(\Omega_L))$ for our problem \eqref{1Lhomeq1}--\eqref{1Lic} (for details we refer the reader to Chapter 4 of \cite{ladyzhenskaia1988linear}). Hence, we conclude that the map $Q$ is well-defined. 
\par Now, we prove that $T$ maps a bounded set to itself. We do so by using energy estimates. We define the weak formulation of \eqref{1Lhomeq1}--\eqref{1Lic} as
\begin{equation}\label{L3e1}
    \int_{\Omega_L} \partial_t p \phi \, dx+ \int_{\Omega_L}\nabla \phi\cdot D^*(v,W) \nabla p  \, dx=  \int_{\Omega_L} \Tilde{f}\phi \, dx
\end{equation}
for all $\phi \in H_0^1(\Omega_L)$. Choosing in \eqref{L3e1} $\phi=p$ , we have 
\begin{equation}\label{L3e2}
    \int_{\Omega_L} \partial_t p p \, dx+ \int_{\Omega_L} \nabla p \cdot D^*(v,W) \nabla p \, dx=  \int_{\Omega_L} \Tilde{f}p \, dx.
\end{equation}
By Lemma \ref{L2} combined with Young's inequality applied to the right-hand side of \eqref{L3e2}, we get
\begin{equation}\label{L3e4}
   \frac{1}{2}\frac{d}{dt} \| \partial_t p \|_{L^2(\Omega_L)}^2+ \alpha\int_{\Omega_L} | \nabla p |^2 \, dx\leq\frac{1}{2} \| \Tilde{f}\|_{L^2(\Omega_L)}^2+\frac{1}{2} \|p\|_{L^2(\Omega_L)}^2,
\end{equation}
and hence,
\begin{equation}\label{L3e5}
   \frac{1}{2}\frac{d}{dt} \| \partial_t p \|_{L^2(\Omega_L)}^2\leq\frac{1}{2} \| \Tilde{f}\|_{L^2(\Omega_L)}^2+\frac{1}{2} \|p\|_{L^2(\Omega_L)}^2.
\end{equation}
 Gronwall's inequality applied to \eqref{L3e5}  guarantees the upper bound
\begin{equation}\label{L3e6}
   \|p\|_{L^{\infty}(0,T;L^2(\Omega_L)}\leq C_1,
\end{equation}
where $C_1:=e^T(\|g\|_{L^2(\Omega_L)}+T\|\Tilde{f}\|_{L^2(\Omega_L)})$ is a positive constant depending on $\Tilde{f}$ and $g$.
Now, integrating \eqref{L3e4} from $0$ to $T$ and using \ref{A3}, \ref{A4} and \eqref{L3e6} we ensure that there exist $C_2>0$ such that

\begin{equation}\label{L3e7}
   \|\nabla p\|_{L^{2}(0,T;L^2(\Omega_L)}\leq C_2,
\end{equation}
where $C_2$ is also depending on $\Tilde{f}$ and $g$.\\ Now, if we take $v\in L^{\infty}((0,T)\times\Omega_L)$ with 
\begin{equation}\label{L3e8}
   \| v\|_{L^{\infty}((0,T)\times(\Omega_L)}\leq \|g\|_{L^{\infty}(\Omega_L) }+T\left(\|f\|_{L^{\infty}(0,T;L^2_{\#}(Z)) }+\|g_N\|_{L^{\infty}(0,T;L^2_{\#}(\Gamma_N)) }\right),
\end{equation}
then by means of methods similar to that ones we used in the proof of Proposition \ref{AT1} in Appendix, we have
\begin{equation}\label{L3e9}
   \| p\|_{L^{\infty}((0,T)\times(\Omega_L)}\leq \|g\|_{L^{\infty}(\Omega_L) }+T\left(\|f\|_{L^{\infty}(0,T;L^2_{\#}(Z)) }+\|g_N\|_{L^{\infty}(0,T;L^2_{\#}(\Gamma_N)) }\right).
\end{equation}

\par We define a new set $S\subset{L^{2}(0,T;L^2(\Omega_L)}$ such that
\begin{equation*}\label{}
   S:=\{u\in {L^{2}(0,T;L^2(\Omega_L)}: \|u\|_{L^{2}(0,T;L^2(\Omega_L)}\leq C_1, \|u\|_{L^{\infty}((0,T)\times(\Omega_L)}\leq M\},
\end{equation*}
where $M:=\|g\|_{L^{\infty}(\Omega_L) }+T\left(\|f\|_{L^{\infty}(0,T;L^2_{\#}(Z)) }+\|g_N\|_{L^{\infty}(0,T;L^2_{\#}(\Gamma_N)) }\right)$.
 From the definition of the map $Q$ together with \eqref{L3e7} and \eqref{L3e9} we note that $Q$ maps the bounded set $S$ into itself. It remains to show that  $S$ is a compact subset of $L^{2}(0,T;L^2(\Omega_L))$. To do so, we prove firstly the following claim:
\newline
 For $v\in S $ there exist constants $C_3,C_4>0$ such that 
\begin{equation}\label{L3e11}
  \| [D^*(v,W)]_{i,j}\|_{L^\infty ((0,T)\times \Omega_L)}\leq C_3+C_4\|P'(v)\|_{L^\infty ((0,T))}.
\end{equation}
Indeed, since $w_i\in H_{\#}^1(Z)$ is a weak solution to the problem \eqref{1Lcell3}--\eqref{1Lcellp},  using \ref{A1}, \ref{A2} together with the Cauchy-Schwarz inequality, and \eqref{bstar},  we obtain
 \begin{align}
    \bigg| \frac{1}{|Z|}\int_Z D(y)e_j\cdot e_i\, dy\bigg|&\leq C, \label{L3e121}\\
    \bigg| \frac{1}{|Z|}\int_Z D(y)\nabla_y w_j\cdot e_i\, dy\bigg|&\leq C, \label{L3e122}\\
   \bigg|  \frac{1}{|Z|}B^*\cdot \int_Z w_j e_i\, dy\bigg|&\leq C,\label{L3e123}\\
   \bigg| \frac{P'(v)}{|Z|}\int_Zw_jB(y)\cdot e_i\, dy\bigg|&\leq C\|P'(v)\|_{L^\infty ((0,T)\times \Omega_L)}\label{L3e124}.
 \end{align}
Combining  \eqref{L3e121}--\eqref{L3e124} with  \eqref{L2e1}, we obtain \eqref{L3e11}
for some positive constants $C_3$ and $C_4$ that can be computed explicitly.

\par   Let $v\in S$, In the weak formulation \eqref{L3e1}.  We choose the test function 
$\phi \in L^2(0,T;H_0^1(\Omega_L))$ such that $\|\phi\|_{H^1(\Omega_L)}\leq 1 $, We get  
\begin{equation}\label{L3e14}
    \int_{\Omega_L} \partial_t p \phi \, dx+ \int_{\Omega_L} \nabla \phi\cdot D^*(v,W) \nabla p  \, dx=  \int_{\Omega_L} \Tilde{f}\phi \, dx.
\end{equation}
Integrating \eqref{L3e14} from $0$ to $T$, we obtain

\begin{equation}\label{L3e15}
\begin{aligned}
      \| \partial_t p \|_{L^2(0,T;[H^1(\Omega_L)]^*)}\leq \sup_{\|\phi\|_{H^1(\Omega_L)}\leq 1} \int_0^T \int_{\Omega_L}-  \nabla \phi\cdot D^*(v,W) \nabla p \, dx\\
    + \sup_{\|\phi\|_{H^1(\Omega_L)}\leq 1}\int_0^T\int_{\Omega_L} \Tilde{f}\phi \, dx.
    \end{aligned}
\end{equation}
Since $v\in S$, using the Cauchy-Schwarz inequality, \eqref{L3e7} and \eqref{L3e11}, we have
\begin{equation}\label{L3e16}
  \sup_{\|\phi\|_{H^1(\Omega_L)}\leq 1} \int_0^T \int_{\Omega_L}- \nabla \phi\cdot D^*(v,W) \nabla p  \, dxdt\leq C.
\end{equation}
By \ref{A3}, \ref{A4} and the Cauchy-Schwarz inequality, we get as well
\begin{equation}\label{L3e17}
  \sup_{\|\phi\|_{H^1(\Omega_L)}\leq 1}\int_0^T\int_{\Omega_L} \Tilde{f}\phi \, dx\leq C.
\end{equation}
From \eqref{L3e15}--\eqref{L3e17}, we finally obtain 
\begin{equation}\label{L3e18}
   \| \partial_t p \|_{L^2(0,T;[H^1(\Omega_L)]^*)}\leq C.
\end{equation}
 Hence we proved that for $v\in S  $ it holds $Q(v)=p\in H^1(0,T;[H^1(\Omega_L)]^*) $. As a direct application of Lions-Aubin's compactness lemma (see \cite{aubin1963analyse}), the space $H^1(0,T;[H^1(\Omega_L)]^*)$ is compactly embedded in  $L^{2}(0,T;L^2(\Omega_L)$. This implies that our set $S$ is a compact subset of $L^{2}(0,T;L^2(\Omega_L)$.
\par  In order to apply the Schauder fixed-point theorem, we still need to prove that $Q$ is continuous on $S$. We guarantee the  continuity property of $Q$ by a sequential argument.
\par Let $v_n\in S$ such that $v_n\rightarrow v$ in $S$ as $n\rightarrow \infty$. We denote $p^n=Q(v_n)$. We prove $p^n\rightarrow p$ as $n\rightarrow \infty$, i.e. we show $ Q(v_n)\rightarrow Q(v)$ as $n\rightarrow \infty$.
\par Using Lions-Aubin's compactness lemma (see \cite{aubin1963analyse}), Banach-Alaglou theorem (see \cite{rudin1973functional}), \eqref{L3e6}, \eqref{L3e7}, \eqref{L3e9} and \eqref{L3e18}, we have
\begin{align}
    v_n&\rightarrow v_0 \hspace{1cm} &\mbox{a.e.}\hspace{1cm} &(0,T)\times \Omega_L \label{c1}\\
    p^n&\rightarrow p\hspace{1cm} &\mbox{in}\hspace{1cm} &L^{2}(0,T;L^2(\Omega_L)\label{c2}\\
    p^n&\rightharpoonup p\hspace{1cm} &\mbox{in}\hspace{1cm} &L^{2}(0,T;H^1(\Omega_L)\label{c3}\\
     \frac{\partial p^n}{\partial x_i} &\rightharpoonup  \frac{\partial p}{\partial x_i}\hspace{1cm} &\mbox{in}\hspace{1cm} &L^{2}(0,T;L^2(\Omega_L)\label{c4}\\
     \frac{\partial p^n}{\partial t} &\rightharpoonup  \frac{\partial p}{\partial t}\hspace{1cm} &\mbox{in}\hspace{1cm} &L^2(0,T;[H^1(\Omega_L)]^*)\label{c5}.
\end{align}
Using \eqref{c5} for $\phi \in H_0^1(\Omega_L)$, we have 

\begin{equation}\label{L3e19}
   \int_{\Omega_L} \partial_t p^n \phi \, dx\rightarrow  \int_{\Omega_L} \partial_t p \phi \, dx
\end{equation}
as $n\rightarrow \infty$.
Since $v_n\rightarrow v$ strongly in $S$ as $n\rightarrow \infty$, we also have
\begin{equation}\label{L3e20}
 \frac{1}{|Z|}  \int_0^T\int_{\Omega_L} P'(v_n)  dxdt\rightarrow \frac{1}{|Z|} \int_0^T \int_{\Omega_L} P'(v) \, dxdt.
\end{equation}

By  \eqref{effectdiff},  \eqref{bstar}, \eqref{c2} and  \eqref{c4}, we obtain

\begin{equation}\label{L3e21}
   \int_{\Omega_L}\nabla \phi\cdot D^*(v_n,W) \nabla p^n  \, dx\rightarrow \int_{\Omega_L}\nabla \phi\cdot D^*(v,W) \nabla p  \, dx,
\end{equation}
as $n\rightarrow\infty$.
Now, applying \eqref{c1}--\eqref{L3e21}, and using the definition of the map $Q$, we conclude that $Q(v_n)\rightarrow Q(v)$ as $n\rightarrow \infty$. Hence $Q$ is sequentially continuous  on $S$.

\par Summarizing, we proved that $S\subset L^{2}(0,T;L^2(\Omega_L)$  is convex, compact , closed and the map $Q:L^2(0,T;L^2(\Omega_L)\rightarrow L^2(0,T;L^2(\Omega_L)$ is continuous on $S$ with $Q(S)\subseteq S$. Schauder's fixed point theorem guarantees that  $Q$ has a fixed point in $S$. This completes the proof of the weak solvability of $P(\Omega_L)$.

\end{proof}
{\subsection{Passage to the limit \texorpdfstring{$L\rightarrow \infty$}{}. Weak solvability of \texorpdfstring{$P(\Omega)$}.{}}

In this section, we prove the existence of weak solutions to  $P(\Omega)$, which is precisely the  upscaled problem derived in the section \ref{upscaled}. We first prove a positivity property as well as a comparison principle for the weak solutions associated to the approximating $P(\Omega_L)$. Relying on these auxilary results, we obtain that the extended solution to $P(\Omega_L)$ converges in a suitable sense to the solution to  $P(\Omega)$  as $L\rightarrow\infty$. Note that here we are approximating the solution of a problem posed in unbounded domain via a monotonically convergent sequence of extended solutions to a problem posed in a bounded domain.

\begin{lemma}[A positivity result]\label{L3}
Let $u_L$ be the weak solution to $P(\Omega)$ in the sense of Definition \ref{D1}. Assume \ref{A1}--\ref{Af} hold true. Then 
\begin{equation*}
    u_L\geq 0\hspace{0.5cm} \mbox{a.e. on} \hspace{0.5cm} (0,T)\times\Omega_L .
\end{equation*}
\end{lemma}
\begin{proof}

We define 
\begin{align}
    u_L^-:&=max\{-u_L,0\}\label{neg}\\
    u_L^+:&=max\{u_L,0\}\label{pos}.
\end{align}
Substituting  $u_L=u_L^+-u_L^-$ in \eqref{wfbd1} and choosing as the test function  $\phi=u_L^-$ we get 
 \begin{multline}\label{ube1}
    \int_{\Omega_L} \partial_t( u_L^+-u_L^-) u_L^- \, dx+ \int_{\Omega_L} \nabla u_L^-\cdot D^*(( u_L^+-u_L^-),W) \nabla ( u_L^+-u_L^-)  \, dx\\
    =  \int_{\Omega_L}\left( \frac{1}{|Z|}\int_Z f\, dy + \frac{-1}{|Z|}\int_{\Gamma_N}g_Nd\sigma_y\right)u_L^- \, dx.
\end{multline}
By \ref{A1}--\ref{A5}, \eqref{neg}, \eqref{pos} on \eqref{ube1}, it yields
\begin{equation}\label{ube2}
    -\frac{1}{2}\frac{d}{dt}\int_{\Omega_L}(u_L^-)^2\, dx-\int_{\Omega_L}\nabla u_L^-\cdot D^*(( u_L^+-u_L^-),W) \nabla (u_L^-)  \, dx\geq 0.
\end{equation}
Recalling Proposition \ref{L2}, we obtain

\begin{equation}\label{ube3}
    \frac{1}{2}\frac{d}{dt}\int_{\Omega_L}(u_L^-)^2\, dx\leq 0.
\end{equation}
Integrating \eqref{ube3} from 0 to $T$ and using assumption \ref{A5}, we conclude 
\begin{equation*}\label{}
    \|u_L^-\|_{L^2((0,T)\times \Omega_L)}^2\leq 0,
\end{equation*}
which yields, $u_L\geq 0$ a.e. on $(0,T)\times \Omega_L$.
\end{proof}
\begin{lemma}[A comparison principle]\label{L4}
Let $v_1,v_2 \in L^2((0,T);H^2(\Omega_L))$ satisfying  \eqref{Lhomeq1} with $0\leq v_1,v_2\leq M$, where M is a positive constant. Further more assume that the inequalities 
\begin{align*}
    v_1(0,x)&\leq v_2(0,x) &\mbox{for}\,\,   &x\in \Omega_L\\
    v_1&\leq v_2 &\mbox{on}\,  &(0,T)\times\partial \Omega_L
\end{align*}
hold. Then
\begin{equation*}
    v_1\leq v_2\hspace{.5cm}\mbox{on}\,  (0,T)\times\overline{\Omega}_L.
\end{equation*}
\end{lemma}
\begin{proof}
To prove Lemma \ref{L4} it is convenient to use a technique of Kirchhoff's transformation (see \cite{bagnall2013application} and \cite{regispredictive}) which help us to transform nonlinearity from diffusion term to time-derivative term. We define Kirchhoff's transformation $\Theta$ as
\begin{equation}\label{l4e1}
    \Theta(u):=\int_0^u D^*(\tau, W)d\tau,
\end{equation}
where $D^*(\cdot, W)$ is defined as in \eqref{effectdiff}. 
From the structure of $D^*$, we get that $D^*(\cdot, W)$ is strictly monotonic and 
\begin{equation}\label{l4e2}
    \mathrm{div}(D^*(u,W)\nabla u)= \Delta \Theta.
\end{equation}
Since $D^*(\cdot, W)$ is strictly monotone (without loss of generality we can assume that $D^*(\cdot, W)$ is increasing) and using \eqref{l4e1}, we see that $\Theta$ is invertible, we denote the inverse of $\Theta$ by $\beta$. Using \eqref{l4e1} and \eqref{l4e2} on \eqref{Lhomeq1}--\eqref{bdic}, we obtain
\begin{align}
    \partial_t \beta (\Theta) -\Delta \Theta &= \frac{1}{|Z|}\int_Z f\, dy + \frac{-1}{|Z|}\int_{\Gamma_N}g_N \, d\sigma_y &\mbox{on} \hspace{.4cm}  (0,T)\times \Omega_L\label{l4e3}\\
    \Theta &=0&\mbox{on} \hspace{.4cm}  (0,T)\times \partial\Omega_L\label{l4e4}\\
    \Theta(0,x)&=g^* &\mbox{on} \hspace{.4cm}  x\in \Omega_L \label{l4e5}
\end{align}
where $g^*:=\Theta(g)$. Now, assume that $\Theta_1,\Theta_2\in L^2((0,T);H^2(\Omega_L))$  satisfy the identity \eqref{l4e3} with 
\begin{equation}\label{l4e6}
    \Theta_1\leq \Theta_2\,\, \mbox{on} \hspace{.4cm}  (0,T)\times \partial\Omega_L.
\end{equation} As a consequence of \eqref{l4e6}, we get 
\begin{equation}\label{neg2}
    (\Theta_2-\Theta_1)^{-}=0\,\,\mbox{on} \hspace{.4cm}  (0,T)\times \partial\Omega_L.
\end{equation}
Consider the identity \eqref{l4e3} for both $\Theta_1$ and $\Theta_2$. Subtract each other, multiply by $\phi\in H_0^1(\Omega_L)$ and integrate over $\Omega_L$ . Later performing integration by parts on the second term , we obtain
\begin{equation}\label{l4e7}
    \int_{\Omega_L} \partial_t(\beta(\Theta_2)-\beta(\Theta_1))\phi dx+\int_{\Omega_L} \nabla(\Theta_2-\Theta_1)\cdot \nabla \phi\, dx=0.
\end{equation} It is convenient to introduce the following  auxiliary problem:
\par Find $G$ satisfying 
\begin{align}
  \Delta G&= (\beta (\Theta_2)-\beta(\Theta_1)) \chi_{\Theta_1\geq \Theta_2}\, &\mbox{on}& \hspace{.4cm}  \Omega_L\label{l4e8}\\
    G&=0\,\, &\mbox{on}& \hspace{.4cm}   \partial\Omega_L\\
    \nabla G(0)&=0 \,\,&\mbox{on}& \hspace{.4cm} \Omega_L\label{l4e10}.
\end{align}
We define the weak form of the problem \eqref{l4e8}--\eqref{l4e10} as
\begin{equation}\label{l4e11}
    \int_{\Omega_L}\nabla G\cdot\nabla \psi \, dx= \int_{\Omega_L} (\beta (\Theta_2)-\beta(\Theta_1)) \chi_{\Theta_1\geq \Theta_2}\,\psi \, dx
\end{equation}
for all $\psi\in H_0^1(\Omega_L).$ For the existence of the weak solution to \eqref{l4e8}--\eqref{l4e10} we refer to \cite{evans2010partial}. Now, we differentiate \eqref{l4e8} with respect to time and consider the associated weak formulation
\begin{equation}\label{l4e12}
    \int_{\Omega_L}\nabla \partial_t G\cdot\nabla \psi \, dx= \int_{\Omega_L}\partial_t( (\beta (\Theta_2)-\beta(\Theta_1)) \chi_{\Theta_1\geq \Theta_2})\,\psi \, dx
\end{equation} for all $\psi\in H_0^1(\Omega_L).$ We substitute $\psi=G$ in  \eqref{l4e12}. Integrating the result with respect to $t$ and using  \eqref{l4e10}, we get
\begin{equation*}\label{}
    \frac{1}{2}\int_{\Omega_L}|\nabla G|^2\,dx=\int_0^t\int_{\Omega_L}\partial_t( (\beta (\Theta_2)-\beta(\Theta_1)) \chi_{\Theta_1\geq \Theta_2})\,G \, dx\,dt.
\end{equation*}
Hence
\begin{equation}\label{l4e14}
    \int_0^t\int_{\Omega_L}\partial_t( (\beta (\Theta_2)-\beta(\Theta_1)) \chi_{\Theta_1\geq \Theta_2})\,G\, dx\,dt\geq 0.
\end{equation}
Choosing $\phi=G$ in \eqref{l4e7} leads to
\begin{equation}\label{l4e15}
    \int_{\Omega_L} \partial_t(\beta(\Theta_2)-\beta(\Theta_1))G dx+\int_{\Omega_L} \nabla(\Theta_2-\Theta_1)\cdot\nabla G\, dx=0.
\end{equation}
Multiply \eqref{l4e8} by $(\Theta_2-\Theta_1)$. After integration by parts and employing \eqref{neg2}, we get
\begin{equation}\label{l4e16}
    \int_{\Omega_L}\nabla G\cdot\nabla(\Theta_2-\Theta_1) \, dx-\int_{\partial \Omega_L}\nabla G\cdot n (\Theta_2-\Theta_1)^+d\sigma_x= \int_{\Omega_L} (\beta (\Theta_2)-\beta(\Theta_1)) \chi_{\Theta_1\geq \Theta_2}\,(\Theta_2-\Theta_1) \, dx.
\end{equation}
Combining \eqref{l4e15} and \eqref{l4e16}, we see that
\begin{equation}\label{l4e17}
    \int_{\Omega_L} \partial_t(\beta(\Theta_2)-\beta(\Theta_1))G dx+\int_{\partial \Omega_L}\nabla G\cdot n (\Theta_2-\Theta_1)^+d\sigma_x+ \int_{\Omega_L} (\beta (\Theta_2)-\beta(\Theta_1)) \chi_{\Theta_1\geq \Theta_2}\,(\Theta_2-\Theta_1) \, dx=0.
\end{equation}
Rearranging suitably \eqref{l4e17}, we have
\begin{multline}\label{l4e18}
    \int_{\Omega_L} \partial_t(\beta(\Theta_2)-\beta(\Theta_1))G(\chi_{\Theta_2\geq \Theta_1}+\chi_{\Theta_2< \Theta_1}) dx+\int_{\partial \Omega_L}\nabla G\cdot n (\Theta_2-\Theta_1)^+\\
    + \int_{\Omega_L} (\beta (\Theta_2)-\beta(\Theta_1)) \chi_{\Theta_1\geq \Theta_2}\,(\Theta_2-\Theta_1) \, dx=0.
\end{multline}
Let $\Omega_L^+$ be the subset of $\Omega_L$  where $\Theta_2<\Theta_1$. Finally based on \eqref{l4e18}, we have 
\begin{equation}\label{l4e19}
    \int_{\Omega_L^+} \partial_t(\beta(\Theta_2)-\beta(\Theta_1))Gdx+ \int_{\Omega_L} (\beta (\Theta_2)-\beta(\Theta_1)) \chi_{\Theta_1\geq \Theta_2}\,(\Theta_2-\Theta_1) \, dx=0.
\end{equation}
We defined $\beta$ as the inverse of strictly increasing function $\Theta$. Consequently, $\beta$ is strictly increasing as well. Integrating \eqref{l4e19} from $0$ to $t$ and using \eqref{l4e14} together with the monotonicity increasing of $\beta$, we obtain 
\begin{equation*}
   \int_0^T \int_{\Omega_L} (\beta (\Theta_2)-\beta(\Theta_1)) (\Theta_2-\Theta_1)^{-} \, dx= 0.
\end{equation*}
This leads to
\begin{equation*}
    (\Theta_2-\Theta_1)^{-}=0 \, \,\,\mbox{a.e} \,\,\,(0,T)\times \Omega_L,
\end{equation*}
and hence,
\begin{equation*}
    \Theta_2\geq \Theta_1 \, \,\,\mbox{a.e} \,\,\,(0,T)\times \Omega_L.
\end{equation*}
If $\Theta_1:=\Theta(v_1)$ and $\Theta_2:=\Theta(v_2)$, then  $\Theta_1,\Theta_2\in L^2((0,T);H^2(\Omega_L))$ are satisfying  \eqref{l4e3} such that 
\begin{equation*}\label{}
    \Theta_1\leq \Theta_2\,\, \mbox{on} \hspace{.4cm}  (0,T)\times \partial\Omega_L.
\end{equation*}
Then
\begin{equation*}
    \Theta(v_2)\geq \Theta(v_1) \, \,\,\mbox{a.e.} \,\,\,(0,T)\times \Omega_L.
\end{equation*}
Since $\beta$ is increasing, it holds
\begin{equation*}
   \beta( \Theta(v_2))\geq \beta(\Theta(v_1)) \, \,\,\mbox{a.e} \,\,\,(0,T)\times \Omega_L.
\end{equation*} We conclude the proof with
\begin{equation*}
   v_2\geq v_1 \, \,\,\mbox{a.e} \,\,\,(0,T)\times \Omega_L.
\end{equation*}
\end{proof}

\par We define the weak solution of the Problem $P(\Omega)$ as follows:
\begin{definition}\label{D2}
    The pair $(u_0,W)\in L^2(0,T;H^1(\mathbb{R}^2))\times [H_\#^1(Z)]^2$ is called weak solution to  $P(\Omega)$ if and only if the following identities are satisfied 
    \begin{align*}
    \int_{\mathbb{R}^2} \partial_t u_0 \phi \, dx+ \int_{\mathbb{R}^2}\nabla \phi\cdot D^*(u_0,W) \nabla u_0  \, dx&=  \int_{\mathbb{R}^2}\left( \frac{1}{|Z|}\int_Z f\, dy + \frac{-1}{|Z|}\int_{\Gamma_N}g_Nd\sigma_y\right)\phi \, dx\label{}\\
    \int_{Z}\nabla_y \psi\cdot D(y)\nabla_y w_i dy+  \mathfrak{A}(u_0)\int_{Z}B(y)\cdot w_i\nabla_y \psi\,  dy&=\int_Z \left(B^*\cdot e_i-B(y)\cdot  \mathfrak{A}(u_0) e_i\right)\psi \, dy\label{}
\end{align*}
for all $(\phi,\psi)\in C_c^{\infty}(\mathbb{R}^2)\times H^1_\#(Z)$ and $a.e.$ $t\in (0,T),$ together with the initial condition
\begin{equation*}\label{}
    u_0(0)=g \hspace{.5cm} \mbox{on} \hspace{.5cm} \mathbb{R}^2.
\end{equation*}
\end{definition}

\begin{theorem}
Assume \ref{A1}--\ref{Af} holds. Then its exists a solution $(u_0,W)$ to $P(\Omega)$ in the sense of Definition \ref{D2}.
\end{theorem}
\begin{proof}
Let $L>0$, Set $\Omega_L$ be a ball centered at origin and having radius $2L$. Assume $(u_L,W)$ be solution of \eqref{Lhomeq1}--\eqref{Lcellp} in the sense of Definition \ref{D1}. We define $\Tilde{u}_L$ as the zero extension of $u_L$ to whole $\mathbb{R}^2$. We show that $\Tilde{u}_L$ converges as $L\rightarrow\infty$ to the solution to  $P(\Omega)$ in the sense of Definition \ref{D2}. 
\par Let $L_2\geq L_1>0$. Take $\Tilde{u}_{L_1}$, $\Tilde{u}_{L_2}$ be the extended solution to  $P(\Omega_{L_1})$ and  $P(\Omega_{L_2})$ respectively. The existence of $\Tilde{u}_{L_1}$ and $\Tilde{u}_{L_2}$ is guaranteed by Theorem \ref{T1}. 
Since ${u}_{L_1}=0$ on $\partial\Omega_{L_1} $ and by Lemma \ref{L3} ${u}_{L_2}\geq 0$ on $\partial\Omega_{L_1} $, so as a result of Lemma \ref{L4} we get $\Tilde{u}_{L_1}\leq\Tilde{u}_{L_2}$ on $\Omega_{L_1}$.  Since $\Tilde{u}_{L_1} $ is the zero extension outside $\Omega_{L_1}$ and $u_{L_2}\geq 0$ by Lemma \ref{L3}, we get 
\begin{equation}\label{T2e1}
    \Tilde{u}_{L_2}\geq \Tilde{u}_{L_1}.
\end{equation}
\par Using the monotone convergence theorem (see Theorem 4 in Appendix E of \cite{evans2010partial}), \eqref{T2e1} and \eqref{L3e9}, we get that the limit 
\begin{align}
    u_\infty :=\lim_{L\rightarrow\infty }\Tilde{u}_{L}
\end{align}
exists in $L^1((0,T)\times \mathbb{R}^2)$. To prove that $\Tilde{u}_{L}$ converges to $u_\infty$ strongly in $L^2((0,T)\times \mathbb{R}^2)$, we use the interpolation inequality. Using the interpolation inequality (see Appendix B of \cite{evans2010partial}) there exists $q\in (2,\infty)$ also $ \alpha\in (0,1)$ such that the following inequality
\begin{equation}\label{T2e3}
    \|\Tilde{u}_{L}-u_\infty\|_{L^2((0,T)\times \mathbb{R}^2)}\leq\|\Tilde{u}_{L}-u_\infty\|_{L^1((0,T)\times \mathbb{R}^2)} ^{1-\alpha}\|\Tilde{u}_{L}-u_\infty\|_{L^q((0,T)\times \mathbb{R}^2)}
\end{equation}
is satisfied.
The first term on the right-hand side of equation \eqref{T2e3} goes to zero as $L\rightarrow \infty$, while the second term in the same equation is bounded. Hence, we conclude that $\Tilde{u}_{L}$ converges to $u_\infty$ strongly in $L^2((0,T)\times \mathbb{R}^2)$. By considering the identity  \eqref{wfbd1} written for $u_L$ and choosing $\phi=u_L$,  we get with the help of Proposition \ref{L2} that there exist $C>0$ such that
\begin{equation}\label{ube5}
    \|\nabla u_L\|_{L^2((0,T)\times \Omega_L)}\leq C.
\end{equation}
The constant  $C$ arising in \eqref{ube5} is independent on $L$. The uniform bound  \eqref{ube5}, ensures 
\begin{equation}\label{ube6}
    \Tilde{u}_L\rightharpoonup u_\infty \,\,\, \mbox{weakly in}\,\,\,L^2((0,T);H^1 (\mathbb{R}^2))
\end{equation}
as $L\rightarrow\infty$.
\par Now we prove that the pair $(\Tilde{u}_{L^*}, W)$ is a solution in the sense of Definition \ref{D2} to the upscaled model $P(\Omega)$. Let $(\phi,\psi)\in C_c^{\infty}(\mathbb{R}^2)\times H^1_\#(Z)$, then there exists a $L^*>0$ such that  $supp(\phi)\subset \Omega_{L^*}$. Then by Theorem \ref{T1}, there exists $(u_{L^*},W)\in L^2(0,T;H^1(\mathbb{R}^2))\times [H_\#^1(Z)]^2$ satisfying 
 \begin{align*}
    \int_{\mathbb{R}^2} \partial_t \Tilde{u}_{L^*} \phi \, dx+ \int_{\mathbb{R}^2} \nabla \phi\cdot D^*(\Tilde{u}_{L^*},W) \nabla \Tilde{u}_{L^*}  \, dx&=  \int_{\mathbb{R}^2}\left( \frac{1}{|Z|}\int_Z f\, dy + \frac{-1}{|Z|}\int_{\Gamma_N}g_Nd\sigma_y\right)\phi \, dx,\label{}\\
    \int_{Z}\nabla_y \psi \cdot D(y)\nabla_y w_idy+  \mathfrak{A}(\Tilde{u}_{L^*})\int_{Z}B(y) \cdot\nabla_y \psi  w_i\,  dy&=\int_Z \left(B^*\cdot e_i-B(y)\cdot  \mathfrak{A}(\Tilde{u}_{L^*}) e_i\right)\psi \, dy.\label{}
\end{align*}
Using the fact that $\Tilde{u}_{L^*}$ converges to $u_\infty$ strongly in $L^2((0,T)\times \mathbb{R}^2)$ and jointly with \eqref{ube6}, we obtain that
\begin{align}
    \int_{\mathbb{R}^2} \partial_t u_\infty \phi \, dx+ \int_{\mathbb{R}^2}  \nabla \phi\cdot D^*(u_\infty,W) \nabla u_\infty \, dx&=  \int_{\mathbb{R}^2}\left( \frac{1}{|Z|}\int_Z f\, dy + \frac{-1}{|Z|}\int_{\Gamma_N}g_Nd\sigma_y\right)\phi \, dx,\label{ube9}\\
    \int_{Z}\nabla_y \psi \cdot D(y)\nabla_y w_idy+ \mathfrak{A}(u_\infty)\int_{Z}B(y)\cdot w_i\nabla_y \psi\,  dy&=\int_Z \left(B^*\cdot e_i-B(y)\cdot \mathfrak{A}(u_\infty) e_i\right)\psi \, dy.\label{ube10}
\end{align}
\eqref{ube9} and \eqref{ube10} point out that  $(u_{\infty},W)\in L^2(0,T;H^1(\mathbb{R}^2))\times [H_\#^1(Z)]^2$ is a solution  defined  in the sense of Definition \ref{D2} to the upscaled model $P(\Omega)$ .
\end{proof}}

\section{Conclusion and outlook}\label{conclusion}

Using the concept of two-scale asymptotic expansion with drift (cf., e.g., \cite{ALLAIRE20102292}), we derived an upscaled equation and corresponding dispersion effective transport tensors for a reaction-diffusion-large drift problem posed in a two-dimensional domain with obstacles. The structure of the resulting upscaled model is a quasi-linear parabolic equation problem (the macroscopic problem) posed on an unbounded domain coupled with a quasi-linear elliptic equation (the cell problem) posed in a bounded domain. The dispersion components in the upscaled model inherit characteristics of both the diffusion and drift exhibited in the original  (microscopic) problem. 

It is worth noting that, in absence of the drift (i.e., for $B=0$), we are essentially performing the classical two-scale asymptotic homogenization  expansions  and consequently, we end up with a previously known macroscopic reaction-diffusion equation. If, on the other hand,  we consider drifts $B$ depending on both variables $x$ and $\frac{x}{\es}$, then the approach based on  asymptotic expansions with drift becomes ineffective. To cope  at least formally with such peculiar situation, the drift $B$ must be additionally assumed to be periodic in his fast variable. 
  If we substitute the drift \eqref{p} with any polynomial of real variables, then the formal homogenization asymptotics still unveils the same upscaled model structure. Note that  \eqref{p} was employed in the proof of solvability of the upscaled problem together with the  classical Schauder fixed-point theorem, a monotonicity trick taken from  \cite{HILHORST20071118}, as well as a suitable reformulation of the problem via a Kirchhoff-type transformation. The major obstacles in proving the existence result were the unboundedness of the macroscopic domain and  the particular macro-micro coupling in the upscaled model. Replacing \eqref{p} with nonlinear versions requires a reconsideration of the functional framework to treat the solvability of the upscaled model equations.  

\par We discussed a two-dimensional problem simply because its origin is intimately  linked to the hydrodynamic limit  of a suitably-scaled totally asymmetric simple exclusion process (TASEP), see \cite{CIRILLO2016436}.  However, the asymptotic expansions work can easily be extended to a 3D case. Note also that instead of involving rectangular obstacles, any shape having a Lipschitz boundary can be taken into consideration without affecting too much the results. Furthermore, if we consider our original microscopic problem posed in bounded domain and consider a slow one directional nonlinear drift, then we obtain the upscaled model derived in \cite{cirillo2020upscaling}. If instead of the bounded domain we consider an infinite strip along the direction of the drift, then a large drift two-scale homogenization is likely to be successful again in recovering a similar result as in this work.  


\par We did not exploit fully the combination of the explosion in the drift and the particular choice of its nonlinear structure. To do so, we would need first to justify rigorously the performed asymptotic expansions, by using e.g. the two-scale convergence with drift as in \cite{hutridurga} and then attempt to prove some sort of corrector estimates, perhaps taking inspiration from \cite{ouaki2012multiscale} and \cite{ouaki2015priori}. 
Progress at this level would facilitate investigations of the same problem now posed for thin layers. This would offer the possibility to combine ``large-drift homogenization" with dimension reduction arguments, which we believe has enormous potential to bring in fundamental understanding what concerns the design of thin composite materials able to endure high-velocity particle impact. 

\section*{Acknowledgments}
The work of V.R. and A.M. is partially supported by the Swedish Research Council's project ``{\em  Homogenization and dimension reduction of thin heterogeneous layers}" (grant nr. VR 2018-03648).

\section*{Appendix}\label{apppendix}
We derive here a particular $L^\infty$-bound on the weak solution of a reaction-diffusion equation similar to our problem. This is an auxiliary ingredient used in the proof of Theorem \ref{T1}.
\begin{proposition}\label{AT1}
Let $T>0$ and $\emptyset\ne \Omega_R\subset\mathbb{R}^2, Z\subset \Omega_R$ with  both $\partial Z$ and $\partial \Omega_R$ having Lipschitz boundaries. \\ Take 
$v\in L^2(0,T;H^1(\Omega_R)\cap L^\infty((0,T)\times \Omega_R)$ with the uniform bound
\begin{equation}\label{A1e1}
   \| v\|_{L^{\infty}((0,T)\times\Omega_R)}\leq \|g\|_{L^{\infty}(\Omega_R) }+T\left(\|f\|_{L^{\infty}(0,T;L^2_{\#}(Z)) }+\|g_N\|_{L^{\infty}(0,T;L^2_{\#}(\Gamma_N)) }\right),
\end{equation}
where $f,g_N $ and $g$ are given functions such that $f\in L^{\infty}(0,T;L^2_{\#}(Z)), g_N\in L^{\infty}(0,T;L^2_{\#}(\Gamma_N)) $ and $g\in L^\infty (\Omega_R)$.  If $u\in L^2(0,T;H^1(\Omega_R))$ is a weak solution to the problem 
  \begin{align}
    \partial_t u +\mathrm{div}( -D^*(v)\nabla_x  u)&=\int_Zf\, dx+\int_{\Gamma_N}g_N\, d\sigma_y  &\mbox{on} \hspace{.4cm}  (0,T)\times \Omega_R, \label{A1e2}\\
    ( -D^*(v)\nabla_x  u)\cdot n&=0&\mbox{on} \hspace{.4cm}  (0,T)\times \Omega_R\\
    u(0,x)&=g(x) &\mbox{for} \hspace{.4cm}  x\in \Omega_R, \label{A1e4}
\end{align}
where $D^*(v)$ is  a  positive definite dispersion tensor  independent of $u$ and linearly dependent on $v$, then $u\in L^\infty((0,T)\times \Omega_R)$ satisfies the  uniform bound 
\begin{equation*}\label{}
   \| u\|_{L^{\infty}((0,T)\times\Omega_R)}\leq \|g\|_{L^{\infty}(\Omega_R) }+T\left(\|f\|_{L^{\infty}(0,T;L^2_{\#}(Z)) }+\|g_N\|_{L^{\infty}(0,T;L^2_{\#}(\Gamma_N)) }\right).
\end{equation*}
\end{proposition}
\begin{proof}
Proceeding similarly as in Lemma 10 from \cite{EDEN2022103408}, we let $h_1\in L^2(0,T;H^1(\Omega_R))$ be the weak solution of 
\begin{align}
    \partial_t h_1 +\mathrm{div}( -D^*(v)\nabla_x  h_1)&=0  &\mbox{on} \hspace{.4cm}  (0,T)\times \Omega_R, \label{A1e6}\\
    ( -D^*(v)\nabla_x  h_1)\cdot n&=0&\mbox{on} \hspace{.4cm}  (0,T)\times \Omega_R\\
    h_1(0,x)&=g(x) &\mbox{for} \hspace{.4cm}  x\in \overline{\Omega}_R, \label{A1e7}
\end{align}
where we define the weak formulation of the problem \eqref{A1e6}--\eqref{A1e7} as: Find $h_1\in L^2(0,T;H^1(\Omega_R))$ satisfying 
\begin{equation}\label{A1e8}
   \int_{\Omega_R} \partial_t h_1\phi\, dx+ \int_{\Omega_R}\nabla \phi\cdot  D^*(v)\nabla_x  h_1\, dx=0
\end{equation}
and 
\begin{equation}\label{A1e9}
   h_1(t=0)=g,
\end{equation}
for all $\phi\in H^1(\Omega_R)$.
The existence of weak solutions to problem \eqref{A1e8}--\eqref{A1e9} follows from the standard theory of parabolic PDE (for details, see for instance \cite{evans2010partial}). Let $M:=\|g\|_{L^\infty(\Omega_R)}$. Then testing suitably
we are led to
\begin{equation}\label{A1e10}
   \int_{\Omega_R} \partial_t (h_1-M)^+\phi\, dx+ \int_{\Omega_R}\nabla \phi\cdot  D^*(v)\nabla_x  (h_1-M)^+\, dx=0,
\end{equation}
with
\begin{equation}\label{A1e11}
   (h_1-M)^+(0,x)=0 \mbox{ for a.e. }  x\in \Omega_R.
\end{equation}
From \eqref{A1e10}--\eqref{A1e11}, we conclude that
\begin{equation}\label{A1e12}
    h_1\leq M \mbox{ a.e. in } (0,T)\times \Omega_R.
\end{equation}
Let $h_2\in L^2(0,T;H^1(\Omega_R))$ be the weak solution to
\begin{align}
    \partial_t h_2 +\mathrm{div}( -D^*(v)\nabla_x  h_2)&=\int_Zf\, dx+\int_{\Gamma_N}g_N\, d\sigma_y  &\mbox{on} \hspace{.4cm}  (0,T)\times \Omega_R, \label{A1e13}\\
    ( -D^*(v)\nabla_x  h_2)\cdot n&=0&\mbox{on} \hspace{.4cm}  (0,T)\times \Omega_R\\
    h_2(0,x)&=0 &\mbox{for} \hspace{.4cm}  x\in \overline{\Omega}_R, \label{A1e15}
\end{align}
Using Duhamel's principle (see e.g Chapter 5 of \cite{Precup+2012}), we can write
\begin{equation}\label{A1ex1}
    h_2(t,x)=\int_0^th_3(\tau,x)d\tau,
\end{equation} where $h_3$ is the solution of 
\begin{align*}
    \partial_t h_3 +\mathrm{div}( -D^*(v)\nabla_x  h_3)&= 0 &\mbox{on} \hspace{.4cm}  (0,T)\times \Omega_R, \label{}\\
    ( -D^*(v)\nabla_x  h_3)\cdot n&=0&\mbox{on} \hspace{.4cm}  (0,T)\times \Omega_R,\\
    h_3(0,x)&=\int_Zf\, dx+\int_{\Gamma_N}g_N\, d\sigma_y &\mbox{for} \hspace{.4cm}  x\in \overline{\Omega}_R. \label{}
\end{align*}
Since  $f\in L^{\infty}(0,T;L^2_{\#}(Z)), g_N\in L^{\infty}(0,T;L^2_{\#}(\Gamma_N)) $, we get
\begin{equation*}\label{}
    \left|\int_Zf\, dx+\int_{\Gamma_N}g_N\, d\sigma_y\right|_{L^\infty((0,T)\times\Omega_L)}\leq \left(\|f\|_{L^{\infty}(0,T;L^2_{\#}(Z)) }+\|g_N\|_{L^{\infty}(0,T;L^2_{\#}(\Gamma_N)) }\right).
\end{equation*}
This gives
\begin{equation}\label{A1e20}
    |h_3|\leq \left(\|f\|_{L^{\infty}(0,T;L^2_{\#}(Z)) }+\|g_N\|_{L^{\infty}(0,T;L^2_{\#}(\Gamma_N)) }\right).
\end{equation}
Inserting \eqref{A1e20} in \eqref{A1ex1}, we obtain
\begin{equation*}\label{}
    |h_2|\leq T\left(\|f\|_{L^{\infty}(0,T;L^2_{\#}(Z)) }+\|g_N\|_{L^{\infty}(0,T;L^2_{\#}(\Gamma_N)) }\right).
\end{equation*}
Now, using the linearity of the diffusion coefficient $D^*(v)$, it yields $u=h_1+h_2$.  Hence, we have 
\begin{equation*}\label{}
    |u|\leq \|g\|_{L^{\infty}(\Omega_R) }+T\left(\|f\|_{L^{\infty}(0,T;L^2_{\#}(Z)) }+\|g_N\|_{L^{\infty}(0,T;L^2_{\#}(\Gamma_N)) }\right)
\end{equation*}
\end{proof}

 \bibliographystyle{amsplain}
	\bibliography{mybib}

\end{document}